\documentclass[a4paper,10pt]{scrartcl}

\usepackage{ucs}
\usepackage[utf8x]{inputenc}
\usepackage{amsmath,amsfonts,amssymb,amsthm}
\usepackage[english]{babel}
\usepackage[OT1]{fontenc}
\usepackage{enumerate}
\usepackage{dsfont,mathtools}
\usepackage[pdftex]{hyperref}
\hypersetup{%
   pdfauthor=Daniel Karrasch,%
   pdftitle=Linearization of Hyperbolic Finite-Time Processes%
}
\usepackage[all]{hypcap}
\usepackage{aliascnt}

\author{Daniel Karrasch\footnote{Fachrichtung Mathematik, Technische Universit\"{a}t Dresden, 01062 Dresden, Germany. E-mail: Daniel.Karrasch@tu-dresden.de} \footnote{This article is based on a part of the author's PhD thesis \cite{Karrasch2012-3}.}}
\title{Linearization of Hyperbolic Finite-Time Processes}
\date{August 27, 2012}

\numberwithin{equation}{section}
\theoremstyle{definition}
\newtheorem{definition}{Definition}[section]

\newaliascnt{example}{definition}
\newaliascnt{theorem}{definition}
\newaliascnt{proposition}{definition}
\newaliascnt{lemma}{definition}
\newaliascnt{corollary}{definition}
\newaliascnt{remark}{definition}
\newaliascnt{claim}{definition}

\theoremstyle{plain}
\newtheorem{theorem}[theorem]{Theorem}
\newtheorem{proposition}[proposition]{Proposition}
\newtheorem{lemma}[lemma]{Lemma}

\newtheorem{corollary}[corollary]{Corollary}

\theoremstyle{remark}
\newtheorem{remark}[remark]{Remark}

\aliascntresetthe{example}
\aliascntresetthe{theorem}
\aliascntresetthe{proposition}
\aliascntresetthe{lemma}
\aliascntresetthe{corollary}
\aliascntresetthe{remark}
\aliascntresetthe{claim}

\renewcommand{\d}{\,\mathrm{d}}
\newcommand{\gap}{\Theta}
\newcommand{\cGL}{GL(n,\R)}
\newcommand{\I}{\mathds{I}}
\newcommand{\J}{\mathds{J}}
\newcommand{\cK}{\mathcal{K}}
\newcommand{\cL}{L}
\newcommand{\cLP}{\mathcal{L}\mathcal{P}}
\newcommand{\N}{\mathds{N}}
\renewcommand{\P}{\mathbb{P}} 
\newcommand{\dP}{\mathds{P}} 
\newcommand{\cP}{\mathcal{P}}
\newcommand{\R}{\mathds{R}}
\renewcommand{\S}{\mathcal{S}}
\newcommand{\Vs}{V^{\textnormal{s}}}
\newcommand{\cVs}{\mathcal{V}^{\textnormal{s}}}
\newcommand{\Vu}{V^{\textnormal{u}}}
\newcommand{\cVu}{\mathcal{V}^{\textnormal{u}}}
\newcommand{\eps}{\varepsilon}
\renewcommand{\phi}{\varphi}
\newcommand{\e}{\mathrm{e}}
\newcommand{\Ineq}{\neq_{\I\times\I}}
\newcommand{\dH}{d_{\textnormal{H}}}
\newcommand{\spec}[1][]{\Sigma^{#1}}
\newcommand{\resolv}[1][]{\rho^{#1}}
\newcommand{\tmin}{t_-}
\newcommand{\tmax}{t_+}
\newcommand{\abs}[1]{\left\lvert#1\right\rvert}
\newcommand{\norm}[1]{\left\lVert#1\right\rVert}
\newcommand{\set}[1]{\left\lbrace#1\right\rbrace}
\newcommand{\Ws}{W^{\textnormal{s}}}
\newcommand{\cWs}{\mathcal{W}^{\textnormal{s}}}
\newcommand{\Wu}{W^{\textnormal{u}}}
\newcommand{\cWu}{\mathcal{W}^{\textnormal{u}}}
\newcommand{\lgr}{\underline{\lambda}}
\newcommand{\ugr}{\overline{\lambda}}
\newcommand{\elgr}[1][]{\underline{\lambda}_{#1}}
\newcommand{\eugr}[1][]{\overline{\lambda}_{#1}}

\DeclareMathOperator{\id}{id}
\DeclareMathOperator{\dist}{dist}

\DeclareMathOperator{\im}{im}
\DeclareMathOperator{\linhull}{span}
\DeclareMathOperator{\Gr}{Gr}
\DeclareMathOperator{\St}{St}
\DeclareMathOperator{\rk}{rk}

\begin{document}

\maketitle

\begin{abstract}
We adapt the notion of processes to introduce an abstract framework for dynamics in finite time, i.e.\ on compact time-sets. For linear finite-time processes a notion of hyperbolicity namely exponential monotonicity dichotomy (EMD) is introduced, thereby generalizing and unifying several existing approaches. We present a spectral theory for linear processes in a coherent way, based only on a logarithmic difference quotient. In this abstract setting we introduce a new topology, prove robustness of EMD and provide exact perturbation bounds. We suggest a new, intrinsic approach for the investigation of linearizations of finite-time processes, including finite-time analogues of the local (un-)stable manifold theorem and theorem of linearized asymptotic stability. As an application, we discuss our results for ordinary differential equations on a compact time-interval.
\end{abstract}

\textbf{Keywords:} Finite-time dynamics, hyperbolicity, spectral theorem, robustness, linearization, stable manifold, nonautonomous differential equations\\

\textbf{MSC:} primary 34A30, 37B55, 37D05, secondary 34D09, 37D10, 37N10

\section{Introduction}

\subsection{Motivation}

In this article we deal with what is informally referred to as \emph{finite-time dynamics} and is usually represented by ordinary differential equations on compact time-intervals, i.e.\
\begin{equation*}
\dot{x} = f(t,x),
\end{equation*}
where, for instance, $f\in C(\I\times\R^n,\R^n)$ and $\I=[\tmin,\tmax]$, $\tmin<\tmax$. The need to analyze such equations arises in many applications such as transport problems in fluid, ocean or atmosphere dynamics (see for instance \cite{Peacock2010} for a recent review), but also increasingly in biological applications (see \cite{Aldridge2006,Rateitschak2010}). There are at least two reasons why one is interested in dynamics on bounded time-sets: one is the interest in transient behavior of solutions although the differential equation might be given on the real (half-)line, and the other one is the simple fact that the right hand side $f$ is given only on a bounded time-set, e.g.\ when it is deduced from observations or measurements. In any case, classical, i.e.\ asymptotic, concepts do not apply directly to the finite-time situation. During the last years efforts were made to establish finite-time analogues to asymptotic notions such as hyperbolicity of trajectories \cite{Haller1998} and linearizations \cite{Berger2009,Rasmussen2010}, Lyapunov exponents \cite{Haller2001,Shadden2005,Haller2011} and stable and unstable manifolds \cite{Haller1998,Haller2001,Duc2008,Berger2011}.

The present article contributes to the investigation of hyperbolicity and its implications in finite time, develops the theory in a coherent way and is organized as follows: after introducing notations in \autoref{sec:notation}, we propose an abstract framework for dynamics on compact time sets $\I$, not necessarily intervals, in \autoref{sec:processes}. For that purpose, there is no need to start the investigation with solution operators generated by ordinary differential equations (ODEs), which are given on an interval containing $\I$. Instead, our framework is based on the well-known notion of processes, thus we state explicitly all required regularity conditions. To the best of our knowledge, such a first-principles approach to finite-time dynamics is proposed for the first time. It prepares the ground for numerical analysis. Starting from a slight modification of the elementary notion of a difference quotient, namely the logarithmic difference quotient, we introduce growth rates as supremum and infimum over the logarithmic difference quotient of subspaces. The step-by-step procedure enables us to prove an important new continuity result, see \autoref{prop:grc}. Another new feature is the introduction of a natural topology on linear finite-time processes, which is consistent with the supremum metric on continuous linear right hand sides of ODEs. The topology allows to consider perturbations of linear processes as in \autoref{cor:grcont}, \autoref{lemma:egrcont} and \autoref{prop:spectrumcont}. In \autoref{sec:hyperbolicity} we develop a spectral theory for linear finite-time processes which is essentially known for solution operators on time-sets consisting of two points \cite{Rasmussen2010} and for intervals \cite{Berger2009,Doan2011}. The underlying hyperbolicity notion is motivated by several approaches \cite{Haller2000,Berger2009,Rasmussen2010,Doan2012}. These are generalized and unified altogether. We call this kind of finite-time hyperbolicity suggestively \emph{exponential monotonicity dichotomy (EMD)}. In \autoref{sec:perturbation} we make use of the introduced topology to establish the robustness of EMD easily. In \autoref{sec:linearization} we introduce formally linearizations of finite-time processes along trajectories and establish finite-time analogues of classical linearization theorems. We integrate the notion of stable and unstable cones firstly introduced in \cite{Doan2011}, but characterize them by bounds on the growth rate functions, and give a direct and intrinsic proof of a Local Stable and Unstable Manifold Theorem as well as a finite-time analogue of the Theorem of Linearized Asymptotic Stability. In \autoref{sec:hartman-grobman} we investigate the local relationship between a process and its linearization, which can be interpreted as a finite-time Hartman-Grobman-like Theorem. In \autoref{sec:applications} we apply the general results of this article to ordinary differential equations on compact time-intervals. It turns out that the concept of finite-time Lyapunov exponents (FTLEs) coincides with the EMD-spectrum on time-sets consisting of two points. Consequently, our definition of spectrum can be used for a generalization of the FTLE-concept to arbitrary compact time-sets.

\subsection{Preliminaries and Notation}
\label{sec:notation}

Throughout this article $\I\subset\R$ denotes a compact subset of the real numbers, where $\tmin\coloneqq\min\I$ and $\tmax\coloneqq\max\I$. We will have occasions to mention the set of all ordered pairs of numbers in $\I$ with unequal components. In accordance with the notion of a relation, we denote by $\Ineq\coloneqq\set{(t,s)\in\I\times\I;~t\neq s}\subset\I\times\I$ the unequal-relation on $\I\times\I$.

For a set $M$ we write $\id_M$ for the identity function on $M$ and $2^M$ for its power set. To distinguish between the evaluation of a function $f\colon A\to B$ at $x\in A$ from the image of a subset $A'\subseteq A$ we write $f(x)$ for the former and $f[A']$ for the latter.

In the following, we consider dynamics on the Banach space $(\R^n,\abs{\cdot})$ with an arbitrary Banach space (vector) norm $\abs{\cdot}$, which in turn induces an operator norm denoted by $\norm{\cdot}$. By $\S$ we denote the unit sphere with respect to the given Banach space norm on $\R^n$. At some point we will make use of the natural Euclidean structure of $\R^n$ to refer to orthogonality. Nevertheless, we do not need specific Hilbert space arguments. As usual, we denote by $\cL(\R^n)$ the set of linear operators on $\R^n$ and by $\cGL$ the subset of invertible operators on $\R^n$.

For a Lipschitz continuous function $f$ we denote by $\abs{f}_{\text{Lip}}$ the Lipschitz constant of $f$. For time-dependent functions, whether vector- or operator-valued, the notation $\norm{\cdot}_{\infty}$ denotes the supremum norm over time with respect to the respective norms. For a continuously differentiable function $f\colon B_0\times\ldots\times B_k\to B$, where $B,B_0,\ldots,B_k$ are (open subsets of) Banach spaces, we denote the partial derivative of $f$ with respect to the $j$-th variable evaluated at $x\in B_0\times\ldots\times B_k$ by $\partial_jf(x)$. Note that the first argument has index $0$. In case $f$ has only one argument, the derivative is also denoted by $f'(x)$ or $\dot{f}(t)$ if the argument is the time variable.

We write $B(x,\delta)$ and $B[x,\delta]$, respectively, for the open and closed ball around $x$ with radius $\delta\in\R_{>0}$.

\section{Finite-Time Processes \& Growth Rates}
\label{sec:processes}

\subsection{Processes}

The notion of a \emph{process} or \emph{two-parameter semi-flow} was originally introduced in \cite{Dafermos1971}. However, we require slightly different conditions. In particular, we include invertibility in the definition.

\begin{definition}[Process, linear process, smooth process]\label{def:linprocess}
We call a continuous function
\begin{align*}
\phi\colon\I\times\I\times\R^n&\to\R^n, & (t,s,x)&\mapsto\phi(t,s,x),
\end{align*}
an \emph{invertible process on $\I$ with state space $\R^n$} if it is Lipschitz continuous in the first argument and if for any $t,s,r\in\I$ and $x\in\R^n$ we have $\phi(t,t,\cdot)=\id_{\R^n}$ and $\phi(t,s,\cdot)\circ\phi(s,r,\cdot)=\phi(t,r,\cdot)$. We denote by $\cP(\I,\R^n)$ the set of invertible processes on $\I$ with state space $\R^n$. We call an invertible process $\Phi\in\cP(\I,\R^n)$ a \emph{linear invertible process on $\I$} if for any $t,s\in\I$ we have $\Phi(t,s,\cdot)\in\cGL$ and the function $t\mapsto\Phi(t,s,\cdot)\in\cL(\R^n)$ is Lipschitz continuous with respect to the operator norm. To emphasize the linearity in the last argument we will write $\Phi(t,s)x$ instead of $\Phi(t,s,x)$ for $t,s\in\I$, $x\in\R^n$. We denote by $\cLP(\I,\R^n)$ the set of linear invertible processes on $\I$ with state space $\R^n$. Let $k\in\N_{>0}$. We call an invertible process $\phi$ a \emph{$C^k$-process on $\I$} if for any $t,s\in\I$ and $x\in\R^n$ we have $\phi(t,s,\cdot)\in C^k(\R^n,\R^n)$ and $\partial_2\phi(\cdot,s,x)\in\cL(\R^n)^{\I}$ is Lipschitz continuous.
\end{definition}

In the following, we always consider invertible processes and hence skip the word invertible. The required Lipschitz continuity of $t\mapsto\Phi(t,\tmin)$ in \autoref{def:linprocess} for a linear process $\Phi$ is rather a technical assumption than an integral part of the notion of a process. In principle, it could be replaced by the weaker, but again technical, assumption of absolute continuity. For convenience, we will assume Lipschitz continuity in the sequel and point out which implications the weaker assumption would have. However, the necessity for a stronger continuity assumption than just continuity will become clear in the course of this article.

We start our investigations with linear processes on $\I$.

\begin{lemma}\label{lemma:uniformconttime}
Let $\Phi\in\cLP(\I,\R^n)$. Then the function
\begin{align*}
\I\times\I&\to\cL(\R^n), & (t,s)&\mapsto\Phi(t,s),
\end{align*}
is uniformly continuous with respect to the (induced) operator norm.
\end{lemma}

\begin{proof}
This is a direct consequence of the uniform continuity of $\Phi\bigr|_{\I\times\I\times B[0,1]}$.
\end{proof}

By the compactness of $\I$ we obtain in the next lemma the continuous dependence of (the norm of) trajectories on the initial value, uniformly in time.

\begin{lemma}\label{lemma:uniformcontspace}
Let $\phi\in\cP(\I,\R^n)$ and $\Phi\in\cLP(\I,\R^n)$. Then the function
\begin{align*}
f\colon\R^n&\to C(\I,\R^n), & x &\mapsto\left(t\mapsto\phi(t,\tmin)x\right),
\end{align*}
is continuous and the functions
\begin{align*}
g\colon\R^n&\to C(\I,\R^n), & x &\mapsto\left(t\mapsto\Phi(t,\tmin)x\right),\\
h\colon\R^n&\to C(\I,\R_{>0}), & x &\mapsto\left(t\mapsto\abs{\Phi(t,\tmin)x}\right)
\end{align*}
are Lipschitz continuous.
\end{lemma}

\begin{proof}
To see the continuity of $f$ in some $x\in\R^n$ it suffices to restrict the continuous function $\phi$ to the compact set $\I\times\I\times B[x,1]\subset\R^{2+n}$, thereby turning it into a uniformly continuous function. The continuity of $f$ in $x$ follows directly from that uniform continuity. As for the continuity of $g$, let $x,y\in\R^n$ and estimate
\[
\norm{g(x)-g(y)}_{\infty}=\sup\set{\abs{\Phi(t,\tmin)(x-y)};~t\in\I}\leq\norm{\Phi(\cdot,\tmin)}_{\infty}\abs{x-y},
\]
where $\norm{\Phi(\cdot,\tmin)}_{\infty}<\infty$ follows from \autoref{lemma:uniformconttime} and the compactness of $\I$. The Lipschitz continuity of $h$ follows from the Lipschitz continuity of the norm.
\end{proof}

Note that \autoref{lemma:uniformconttime} and \autoref{lemma:uniformcontspace} did not use the Lipschitz continuity assumption in \autoref{def:linprocess}. Even stronger, the continuity of $x\mapsto\left(t\mapsto\phi(t,\tmin)x\right)$ relies only on the continuity of $\phi$ as a function defined on $\I\times\I\times\R^n$ and mapping to $\R^n$. Although seemingly simple, \autoref{lemma:uniformcontspace} indicates an important difference in asymptotic and finite-time analysis: on unbounded (to the right) time-sets the continuous dependence of whole trajectories on the initial value corresponds to the definition of stability in the sense of Lyapunov, which is a nontrivial feature of certain trajectories. In the finite-time case, as \autoref{lemma:uniformcontspace} shows, it holds under very general assumptions.

\subsection{Logarithmic Difference Quotient}

The following concept will help us to introduce some of the forthcoming notions and to present the theory in an elegant and coherent way.

\begin{definition}[Logarithmic difference quotient]\label{def:logdiffquot}
We define
\begin{align*}
\Delta_{\I}\colon C(\I,\R_{>0}) &\to C(\Ineq,\R), & f&\mapsto\left((t,s)\mapsto\frac{\ln f(t)-\ln f(s)}{t-s}\right),
\end{align*}
and we call $\Delta_{\I}(f)(t,s)$ the \emph{logarithmic difference quotient of $f$ at $t$ and $s$} for $f\in C(\I,\R_{>0})$ and $t,s\in\Ineq$. For notational convenience we will write $\Delta$ for $\Delta_{\I}$ when there is no risk of confusion.
\end{definition}

\begin{remark}\label{remark:boundedlog}
Note that due to compactness of $\I$ any function $f\in C(\I,\R_{>0})$ is uniformly continuous and $f[\I]\subseteq[a,b]$ with $0<a<b$, i.e.\ $f$ is bounded above and bounded away from zero. Furthermore, $\ln\bigr|_{[a,b]}$ is Lipschitz continuous with Lipschitz constant $\frac{1}{a}$ and bounded. Suppose $f$ is additionally Lipschitz continuous, then
\begin{equation*}
\abs{\sup\set{\Delta(f)(t,s);~(t,s)\in\Ineq}},\abs{\inf\set{\Delta(f)(t,s);~(t,s)\in\Ineq}}\leq\frac{\abs{f}_{\text{Lip}}}{a}.
\end{equation*}
Furthermore, we have that $\Delta(\cdot)(t,s)$ is continuous for any $(t,s)\in\Ineq$. Summarizing, $\Delta$ considered as
\begin{align*}
\Delta\colon C(\I,\R_{>0})\times\Ineq&\to\R, & (f,t,s)&\mapsto\frac{\ln f(t)-\ln f(s)}{t-s},
\end{align*}
is continuous in the first argument and jointly continuous in the last two arguments, but, in general, not jointly continuous in all three arguments. However, for a linear process $\Phi\in\cLP(\I,\R^n)$ one can show that the family of functions $(\Delta(\abs{\Phi(\cdot,\tmin)x}))_{x\in\S}$ is equicontinuous. This can be calculated directly or, alternatively, can be deduced from \autoref{lemma:uniformcontspace} and the Theorem of Arzel\`{a}-Ascoli.
\end{remark}

When applied to a linear process we recover joint continuity for the logarithmic difference quotient.

\begin{lemma}\label{lemma:Deltacont}
Let $\Phi\in\cLP(\I,\R^n)$. Then
\begin{align*}
\R^n\setminus\set{0}\times\Ineq&\to\R, & (x,t,s)&\mapsto\Delta(\abs{\Phi(\cdot,\tmin)x})(t,s),
\end{align*}
is continuous.
\end{lemma}

Note that the function defined in \autoref{lemma:Deltacont} is well-defined, since we assume the linear process to be invertible (no trajectory starting in $\R^n\setminus\set{0}$ at $\tmin$ attains zero on $\I$).

\begin{proof}
Let $x\in\R^n\setminus\set{0}$, $(t,s)\in\Ineq$ and $\eps\in\R_{>0}$. By equicontinuity with respect to the initial value there exists $\delta_1\in\R_{>0}$ such that for all $y\in\R^n\setminus\set{0}$
\begin{equation*}
\max\set{\abs{t-t'},\abs{s-s'}}<\delta_1~\Rightarrow~\abs{\Delta(\abs{\Phi(\cdot,\tmin)y})(t,s)-\Delta(\abs{\Phi(\cdot,\tmin)y})(t',s')}<\frac{\eps}{2}.
\end{equation*}
By \autoref{lemma:uniformcontspace} and hence continuity of $\R^n\setminus\set{0}\ni y\mapsto\Delta(\abs{\Phi(\cdot,\tmin)y})(t,s)$ there exists $\delta_2\in\R_{>0}$ such that for all $x'\in\R^n\setminus\set{0}$
\begin{equation*}
\abs{x-x'}<\delta_2~\Rightarrow~\abs{\Delta(\abs{\Phi(\cdot,\tmin)x})(t,s)-\Delta(\abs{\Phi(\cdot,\tmin)x'})(t,s)}<\frac{\eps}{2}.
\end{equation*}
Combining these two estimates, we obtain for $(t',s')\in\Ineq$, $x'\in\R^n\setminus\set{0}$
\begin{align*}
\abs{\Delta(\abs{\Phi(\cdot,\tmin)x})(t,s)-\Delta(\abs{\Phi(\cdot,\tmin)x'})(t',s')}<\eps,
\end{align*}
whenever $\max\set{\abs{t-t'},\abs{s-s'},\abs{x-x'}}<\min\set{\delta_1,\delta_2}$.
\end{proof}

\begin{remark}
The logarithmic difference quotient can be regarded as a \emph{finite-time exponential growth rate} as introduced in \cite{Colonius2008}. With the notation used there, we find that
\begin{equation*}
\lambda^{(t-s)}(s,\Phi(s,\tmin)x)=\Delta(\abs{\Phi(\cdot,\tmin)x})(t,s).
\end{equation*}
\end{remark}

\autoref{def:logdiffquot} and \autoref{lemma:Deltacont} will have a major impact on the following theory and allow a development of the known finite-time spectral theory, starting from the notion of the logarithmic difference quotient alone. Another useful observation is the following, which holds obviously by linearity of $\Phi(t,s)$ and \autoref{def:logdiffquot}.

\begin{lemma}\label{lemma:constlines}
Let $\Phi\in\cLP(\I,\R^n)$. Then for any $x\in\R^n\setminus\set{0}$ and $\lambda\in\R\setminus\set{0}$ one has
\begin{equation*}
\Delta(\abs{\Phi(\cdot,\tmin)(\lambda x)})=\Delta(\abs{\Phi(\cdot,\tmin)x}).
\end{equation*}
\end{lemma}

As a consequence, we can equally well define $\Delta(\abs{\Phi(\cdot,\tmin)\cdot})$ on the quotient set $\P^{n-1}$, the (real) projective space obtained by identifying vectors $x,y\in\R^n\setminus\set{0}$ if they are nontrivial multiples of each other. The projective space $\P^{n-1}$ can be identified to the set of one-dimensional subspaces, or lines through the origin, in $\R^n$. A generalization of that concept is the Grassmann manifold $\Gr(k,\R^n)$ of $k$-dimensional subspaces of $\R^n$, which will turn out to be very useful. Let us briefly recall the construction and its basic properties (following essentially \cite{Ferrer1994,Absil2004}).

For any $k\in\set{1,\ldots,n}$ one can introduce the Grassmann manifold $\Gr(k,\R^n)$ as follows: consider the set of all orthonormal $k$-frames in $\R^n$, which is also referred to as the Stiefel manifold $\St(k,\R^n)$, i.e.\ $k$-tuples of orthonormal vectors in $\R^n$ represented by $n\times k$-matrices $A$ such that $A^{\top}A=\id$. It is well-known that the Stiefel manifold is a compact subset of $\R^{n\times k}$. The Grassmann manifold can then be obtained as the image of the function $\pi\colon\St(k,\R^n)\to\Gr(k,\R^n)$ mapping $A=\begin{pmatrix} a_1 & \cdots & a_k\end{pmatrix}$ to $\linhull\set{a_1,\ldots,a_k}$. Endowing $\Gr(k,\R^n)$ with the final topology with respect to $\pi$ turns it into a compact set. As is proved in \cite{Ferrer1994}, the so-called gap metric $\gap\colon\Gr(k,\R^n)^2\to\R_{\geq 0}$, $\gap(L,M)=\norm{\pi_L-\pi_M}$, where $\norm{\cdot}$ denotes the operator norm and $\pi_X$ denotes the orthogonal projection onto a subspace $X$, induces a topology that coincides with the final topology of $\pi$. Furthermore, the same topology is induced by the Hausdorff distance between the intersections of, respectively, the subspaces $L$ and $M$ with the unit sphere (see, for instance, \cite[Chapter 13]{Gohberg2006}). As a consequence, $X_n\xrightarrow{n\to\infty}X$ for $X\in\Gr(k,\R^n)$ and $(X_n)_{n\in\N}\in(\Gr(k,\R^n))^{\N}$ implies that for any $x\in X$ there exists a sequence $(x_n)_{n\in\N}\in\S^{\N}$ with $x_n\in X_n$ for each $n\in\N$ such that $x_n\xrightarrow{n\to\infty}x$ with respect to any norm on $\R^n$ by the equivalence of norms in finite-dimensional normed vector spaces.

As a consequence of \autoref{lemma:constlines}, we find that $x\mapsto\Delta(\abs{\Phi(\cdot,\tmin)x})(t,s)$ is actually uniformly continuous for any $(t,s)\in\Ineq$. Since many of the following results are concerned with continuity of functions derived from the logarithmic difference quotient, we put the following additional assumption on the linear processes which are considered in this work.\\

\textbf{General Assumption:} Throughout, we assume that for any $\Phi\in\cLP(\I,\R^n)$ the function
\begin{align*}
\S\times\Ineq&\to\R, & (x,t,s)&\mapsto\Delta(\abs{\Phi(\cdot,\tmin)x})(t,s)
\end{align*}
can be extended continuously to $\S\times\overline{\Ineq}$, where $\overline{\Ineq}$ denotes the closure of $\Ineq$ in $\R^2$.

As a consequence, $\overline{\Ineq}$ is compact and the extended function is uniformly continuous. The following remark discusses the consequences of our general assumption.

\begin{remark}\label{remark:generalassumption}
As $\I$ is compact, so is $\I\times\I\subset\R^2$. Clearly, $\Ineq$ is still bounded, but not necessarily closed. We need to distinguish two cases: either $\I$ has no limit points, i.e. $\I$ is finite, or $\I$ has limit points, e.g.\ when $\I$ is a (nontrivial) compact interval.

\begin{enumerate}[(a)]
\item If $\I$ does not have any limit points then so does $\Ineq$. Thus, $\Ineq$ is closed already and the general assumption does not pose any restrictions on the finite-time processes under consideration.
\item Let $t^*\in\I$ be a limit point. Then our general assumption requires that the limit
\[
\lim_{t,s\to t^*}\frac{\ln\abs{\Phi(t,\tmin)x}-\ln\abs{\Phi(s,\tmin)x}}{t-s}=\frac{\abs{\Phi(\cdot,\tmin)x}'(t^*)}{\abs{\Phi(t^*,\tmin)x}}
\]
exists for any $x\in\S$ and is continuous in $x$. A sufficient (but not necessary) condition for the derivative on the right-hand side to exist is the continuous differentiability of the norm (not at $0$, of course) and the differentiability of the trajectories $t\mapsto\Phi(t,\tmin)x$ for any $x\in\S$ at the (time-) limit point $t^*$. Clearly, solution operators of linear ordinary differential equations with continuous right-hand side as usually considered together with continuously differentiable norms do satisfy our general assumption.
\end{enumerate}
\end{remark}

The following simple observation associates the notions introduced in the present work to notions with the same nomenclature appearing in the references \cite{Berger2009,Doan2011,Doan2012}.

\begin{lemma}\label{lemma:equivalences}
Let $f\in C(\I,\R_{>0})$. Then the following statements are equivalent:
\begin{enumerate}[(i)]
\item There exists $\delta\in\R_{>0}$ such that for any $t,s\in\I$, $t\geq s$, one has $f(t)\leq\e^{-\delta(t-s)}f(s)$.
\item There exists $\delta\in\R_{>0}$ such that $t\mapsto\e^{\delta t}f(t)$ is decreasing.
\item $\sup\set{\Delta(f)(t,s);~(t,s)\in\Ineq}<0$.
\end{enumerate}
Moreover, if $\I$ is an interval and $f$ is differentiable, then each of the above statements is equivalent to
\begin{enumerate}[(i)]
\setcounter{enumi}{3}
\item there exists $\delta\in\R_{>0}$ such that $f'\leq-\delta f$ holds.
\end{enumerate}
\end{lemma}

Analogue statements hold for $f\in C(\I,\R_{>0})$ if $\inf\set{\Delta(f)(t,s);~(t,s)\in\Ineq}>0$.

\subsection{Growth Rates}
\label{sec:growthrates}

Next, we introduce \emph{growth rates}, functions that describe the dynamical behavior of subspaces under the linear process $\Phi$. With \autoref{lemma:equivalences} one easily establishes equivalence to the growth rates as defined in \cite{Berger2009,Doan2011,Doan2012}.

\begin{definition}[{Growth rate, cf.\ \cite{Berger2009,Doan2011,Doan2012}}]\label{def:growthrate}
We define
\begin{align*}
\lgr^{\I}\colon\bigcup_{k=1}^n \Gr(k,\R^n)\times\cLP(\I,\R^n)&\to\R, & (X,\Phi)&\mapsto\min_{\substack{x\in X\cap\S\\(t,s)\in\overline{\Ineq}}}\set{\Delta(\abs{\Phi(\cdot,\tmin)x})(t,s)},\\
\ugr^{\I}\colon\bigcup_{k=1}^n \Gr(k,\R^n)\times\cLP(\I,\R^n)&\to\R, & (X,\Phi)&\mapsto\max_{\substack{x\in X\cap\S\\(t,s)\in\overline{\Ineq}}}\set{\Delta(\abs{\Phi(\cdot,\tmin)x})(t,s)},
\end{align*}
and call $\lgr^{\I}(X,\Phi)$ and $\ugr^{\I}(X,\Phi)$, respectively, the \emph{lower} and \emph{upper growth rate of $X$ under $\Phi$}\index{growth rate!lower}\index{growth rate!upper}. For convenience, we drop the index $\I$ in case the time set is clear from the context. We extend the definition naturally by $\lgr(\set{0},\Phi) = \infty$ and $\ugr(\set{0},\Phi) = -\infty$ for any $\Phi\in\cLP(\I,\R^n)$.
\end{definition}

\begin{remark}\label{remark:Bohl}
One can consider the $1$-dimensional growth rates as finite-time analogues to \emph{Bohl exponents}, see \cite{Bohl1913} and also \cite[p.\ 118, pp.\ 146--148]{Daletskij1974} for the definition and a historical review. Furthermore, one readily verifies that for any $\Phi\in\cLP(\I,\R^n)$ and $X\in\Gr(1,\R^n)$ one has $\abs{\lgr(X,\Phi)},\abs{\ugr(X,\Phi)}\leq L_1L_2<\infty$, where $L_1\in\R_{>0}$ is the Lipschitz constant of the logarithm restricted to some interval as in \autoref{remark:boundedlog} and $L_2\in\R_{>0}$ is the (global) Lipschitz constant of $t\mapsto\Phi(t,s,\cdot)\in\cL(\R^n)$, cf.\ \autoref{def:linprocess}. The boundedness of growth rates for nontrivial subspaces does not hold in general when we weaken the assumed Lipschitz continuity of linear processes to absolute continuity, see \cite{Berger2009}.
\end{remark}

The following simple observations follow directly from the definition.

\begin{lemma}[{cf.\ \cite[Remark 6]{Doan2011}}]\label{lemma:gr}
Let $\Phi\in\cLP(\I,\R^n)$ and $X,Y\subseteq\R^n$ be subspaces. Then
\begin{enumerate}[(i)]
\item $X\subseteq Y$ implies $\lgr(Y,\Phi)\leq\lgr(X,\Phi)$ and $\ugr(Y,\Phi)\geq\ugr(X,\Phi)$;
\item $X\cap Y\neq\set{0}$ implies $\lgr(X,\Phi)\leq\ugr(Y,\Phi)$ and $\lgr(Y,\Phi)\leq\ugr(X,\Phi)$.
\end{enumerate}
\end{lemma}

\begin{proof}
(i) This is clear since, with respect to $Y$, infimum and supremum are taken over a larger set, respectively.

(ii) Suppose $X\cap Y\neq\set{0}$ and let $x\in(X\cap Y)\setminus\set{0}$. Then by definition we have for any $(t,s)\in\Ineq$
\begin{align*}
\lgr(X,\Phi)&\leq\Delta(\abs{\Phi(\cdot,\tmin)x})(t,s)\leq\ugr(Y,\Phi)
\intertext{and}
\lgr(Y,\Phi)&\leq\Delta(\abs{\Phi(\cdot,\tmin)x})(t,s)\leq\ugr(X,\Phi).\qedhere
\end{align*}
\end{proof}

Next, we establish continuity for the growth rates when restricted to some $\Gr(k,\R^n)$. This extends \cite[Theorem 9]{Doan2011} and will be used as an important tool at several points in this work.

\begin{proposition}\label{prop:grc}
Let $\Phi\in\cLP(\I,\R^n)$. Then for each $k\in\set{1,\ldots,n}$, the restrictions
\begin{equation*}
\lgr(\cdot,\Phi)\bigr|_{\Gr(k,\R^n)}\quad\text{and}\quad\ugr(\cdot,\Phi)\bigr|_{\Gr(k,\R^n)}
\end{equation*}
are continuous and bounded.
\end{proposition}

\begin{proof}
We prove only the lower growth rate case. Let $\eps\in\R_{>0}$. By the general assumption, we have that
\[
\S\times\overline{\Ineq}\ni(x,t,s)\mapsto\Delta(\abs{\Phi(\cdot,\tmin)x})(t,s)\in\R
\]
is uniformly continuous. Hence, there exists $\delta\in\R_{>0}$ such that for any $x_1,x_2\in\S$, $t_1,t_2,s_1,s_2\in\I$ the inequality
\[
\max\set{\abs{x_1-x_2},\abs{t_1-t_2},\abs{s_1-s_2}}<\delta
\]
implies
\[
\abs{\Delta(\abs{\Phi(\cdot,\tmin)x_1})(t_1,s_1)-\Delta(\abs{\Phi(\cdot,\tmin)x_2})(t_2,s_2)}<\eps.
\]
Now, let $X,Y\in\Gr(k,\R^n)$ with $\Theta(X,Y)<\delta/2$ and $\lgr(X,\Phi)=\Delta(\abs{\Phi(\cdot,\tmin)x_1})(t_1,s_1)$ and $\lgr(Y,\Phi)=\Delta(\abs{\Phi(\cdot,\tmin)x_2})(t_2,s_2)$, $x_1\in X\cap\S$, $x_2\in Y\cap\S$, $t_1,t_2,s_1,s_2\in\I$. By \cite[p. 198, Eq.\ (2.12)]{Kato1995} there exist $\hat{x}_1\in Y\cap\S$ and $\hat{x}_2\in X\cap\S$ such that $\abs{x_1-\hat{x}_1},\abs{x_2-\hat{x}_2}<\delta$. Finally, to prove that $\abs{\lgr(Y,\Phi)-\lgr(X,\Phi)}<\eps$, we consider two cases: if $\lgr(Y,\Phi)\geq\lgr(X,\Phi)$ we have
\begin{align*}
\abs{\lgr(Y,\Phi)-\lgr(X,\Phi)} &=\Delta(\abs{\Phi(\cdot,\tmin)x_2})(t_2,s_2)-\Delta(\abs{\Phi(\cdot,\tmin)x_1})(t_1,s_1)\\
&\leq\Delta(\abs{\Phi(\cdot,\tmin)\hat{x}_1})(t_1,s_1)-\Delta(\abs{\Phi(\cdot,\tmin)x_1})(t_1,s_1)<\eps,
\end{align*}
if $\lgr(Y,\Phi)<\lgr(X,\Phi)$ we have
\begin{align*}
\abs{\lgr(Y,\Phi)-\lgr(X,\Phi)} &=\Delta(\abs{\Phi(\cdot,\tmin)x_1})(t_1,s_1)-\Delta(\abs{\Phi(\cdot,\tmin)x_2})(t_2,s_2)\\
&\leq\Delta(\abs{\Phi(\cdot,\tmin)\hat{x}_2})(t_2,s_2)-\Delta(\abs{\Phi(\cdot,\tmin)x_2})(t_2,s_2)<\eps,
\end{align*}
which finishes the proof.
\end{proof}

For later robustness investigations the following (semi-)metric and the in\-duced topology on $\cLP(\I,\R^n)$ will turn out to be very helpful.

\begin{definition}[{Metric on $\cLP(\I,\R^n)$}]\label{def:metric}
We define $\tilde{d}_{\I},d_{\I}\colon\cLP(\I,\R^n)^2\to\R_{\geq 0}$ by
\begin{align*}
\tilde{d}_{\I}\colon(\Phi,\Psi)&\mapsto\sup_{X\in\Gr(1,\R^n)}\Bigl\{\max\set{\abs{\lgr(X,\Phi)-\lgr(X,\Psi)},\abs{\ugr(X,\Phi)-\ugr(X,\Psi)}}\Bigr\},\\
d_{\I}\colon(\Phi,\Psi)&\mapsto\max\set{\sup_{x\in\S}\norm{(\Phi(\cdot,\tmin)-\Psi(\cdot,\tmin))x}_{\infty},\tilde{d}_{\I}(\Phi,\Psi)}.
\end{align*}
Obviously, $\tilde{d}_{\I}$ is a semimetric, i.e.\ it satisfies non-negativity, symmetry and the triangle inequality (but not definiteness), and $d_{\I}$ is a metric on $\cLP(\I,\R^n)$.
\end{definition}

Note that, in general, $d_{\I}$ cannot be extended to a proper metric on the set of absolutely continuous linear processes due to possibly unbounded growth rates. Nevertheless, the (semi-)metric can be used to define open balls around absolutely continuous linear processes.

In the following, we endow $\cLP(\I,\R^n)$ with the topology induced by $d_{\I}$. By construction of the topology we obtain the continuous dependence of growth rates on the linear process.

\begin{corollary}\label{cor:grcont}
Let $X\subseteq\R^n$ be a subspace. Then $\lgr(X,\cdot),\ugr(X,\cdot)\colon\cLP(\I,\R^n)\to\R\cup\set{-\infty,\infty}$ are Lipschitz continuous with Lipschitz constant $1$.
\end{corollary}

\subsection{Extremal Growth Rates}

The following notion is connected to subspaces with optimal growth rates among subspaces of the same dimension. It will play an important role in the linear spectral theory as well as in robustness issues.

\begin{definition}[{Extremal $k$-growth rates, cf.\ \cite{Berger2009,Doan2011,Doan2012}}]\label{def:extgrowthrate}
For $k\in\set{1,\ldots,n}$ we define
\begin{align*}
\elgr[k]^{\I}\colon\cLP(\I,\R^n)&\to\R, & \Phi&\mapsto\max\lgr^{\I}([\Gr(k,\R^n)],\Phi),\\
\eugr[k]^{\I}\colon\cLP(\I,\R^n)&\to\R, & \Phi&\mapsto\min\ugr^{\I}([\Gr(k,\R^n)],\Phi),
\end{align*}
and call $\elgr[k]^{\I}(\Phi)$ and $\eugr[k]^{\I}(\Phi)$, respectively, the \emph{maximal lower} and \emph{minimal upper $k$-growth rate of $\Phi$}. For convenience, we drop the index $\I$ in case the time set is clear from the context. We extend the above definition naturally by $\elgr[0]^{\I}(\Phi)=\infty$ and $\eugr[0]^{\I}(\Phi)=-\infty$.
\end{definition}

\begin{remark}\label{remark:maxsubspace}
Note that the extremal $k$-growth rate functions are well-defined due to the continuity from \autoref{prop:grc} and the compactness of $(\Gr(k,\R^n),\gap)$. That means in particular that for any $\Phi\in\cLP(\I,\R^n)$ and $k\in\set{0,\ldots,n}$ there exist $X,Y\in\Gr(k,\R^n)$ such that $\lgr(X,\Phi)=\elgr[k](\Phi)$ and $\ugr(Y,\Phi)=\eugr[k](\Phi)$. We refer to such subspaces as \emph{extremal subspaces}. Of course, in general, extremal subspaces need not be unique.
\end{remark}

\begin{lemma}[{cf.\ \cite[Remark 11]{Berger2009}, \cite[Remark 9]{Doan2011}}]\label{lemma:orderedgrowthrates}
For any $\Phi\in\cLP(\I,\R^n)$, the extremal growth rates are ordered and nested as follows:
\begin{multline*}
-\infty=\eugr[0](\Phi)<\elgr[n](\Phi)\leq\eugr[1](\Phi)\leq\elgr[n-1](\Phi)\leq\ldots\\
\ldots\leq \eugr[n-1](\Phi)\leq\elgr[1](\Phi)\leq\eugr[n](\Phi)<\elgr[0](\Phi)=\infty.
\end{multline*}
\end{lemma}

\begin{proof}
The ordering properties follow directly from \autoref{lemma:gr}(i) and the nesting property from \autoref{lemma:gr}(ii).
\end{proof}

Another observation is the continuity of the extremal growth rate functions, which easily follows from \autoref{cor:grcont}.

\begin{lemma}\label{lemma:egrcont}
For any $k\in\set{1,\ldots,n}$ the extremal $k$-growth rate functions $\elgr[k],\eugr[k]$ are Lipschitz continuous with Lipschitz constant $1$.
\end{lemma}

Note that \cite[Theorem 20]{Doan2012} is essentially the $\eps$-$\delta$-notation of the continuity established in \autoref{lemma:egrcont} for the special case that $\I$ is finite.

More interesting is the following convergence result with respect to the time-set. We denote by $\dH$ the Hausdorff metric on the space of compact subsets of $M\subseteq\R$, denoted by $\cK(M)$. The proof in \cite{Doan2012} applies also to our more general situation here.

\begin{lemma}[{\cite[Theorem 17]{Doan2012}}]\label{lemma:partialcont}
Let $\Phi\in\cLP(\I,\R^n)$. Then for each $k\in\set{1,\ldots,n}$ and any sequence $\J=\left(\J_i\right)_{i\in\N}\in\cK(\I)^{\N}$ of compact subsets of $\I$ with $\dH(\I,\J_i)\xrightarrow{i\to\infty}0$ one has
\begin{align*}
\lim_{i\to\infty} \abs{\elgr[k]^{\J_i}(\Phi\bigr|_{\J_i^2}) -\elgr[k]^{\I}(\Phi)}&=0, & \text{and} && \lim_{i\to\infty} \abs{\eugr[k]^{\J_i}(\Phi\bigr|_{\J_i^2}) -\eugr[k]^{\I}(\Phi)}&=0.
\end{align*}
\end{lemma}

\section{Spectral Theory for Linear Finite-Time Processes}
\label{sec:hyperbolicity}

This section is devoted to the development of a spectral theory for linear finite-time processes, which is in the spirit of \cite{Sacker1978,Siegmund2002,Berger2009}. We start with a notion based on a certain dynamical dichotomy. We denote by $\dP(\R^n)$ the set of (linear) projections on $\R^n$. Throughout this section let $\Phi\in\cLP(\I,\R^n)$ denote a linear process on $\I$.

\subsection{Exponential Monotonicity Dichotomy}
\label{sec:emd}

\begin{definition}[Exponential monotonicity dichotomy]\label{def:EMD}
$\Phi$ admits an \emph{exponential mo\-no\-to\-ni\-city dichotomy (EMD) on $\I$} if there exists $k\in\set{0,\ldots,n}$ such that
\begin{equation}\label{eq:defhyp1}
\eugr[k](\Phi)<0<\elgr[n-k](\Phi).
\end{equation}
For brevity, we sometimes call a linear process $\Phi$ on $\I$ admitting an EMD also \emph{(finite-time) hyperbolic}. We call a hyperbolic linear process $\Phi$ on $\I$ \emph{(finite-time) attractive} if $k=n$ and \emph{(finite-time) repulsive} if $k=0$.
\end{definition}

\begin{remark}
Since the logarithmic difference quotient depends on the norm, so do the growth rates and hence the property of admitting an EMD. In general, it might be possible that a linear process on $\I$ is hyperbolic with respect to one norm, but is not with respect to another norm; see \cite[Example 14]{Berger2009}.
\end{remark}

In the next lemma, we characterize an EMD on $\I$.

\begin{lemma}[{cf.\ \cite{Berger2009}}]\label{lemma:hypgrowth}
The following statements are equivalent:
\begin{enumerate}[(i)]
\item $\Phi$ admits an EMD on $\I$ with $k\in\set{0,\ldots,n}$.
\item There exist subspaces $X\in\Gr(k,\R^n)$ and $Y\in\Gr(n-k,\R^n)$ such that $\ugr(X,\Phi)<0<\lgr(Y,\Phi)$.
\item There exists a projection $Q\in\dP(\R^n)$ with $\rk Q=k$ such that $\ugr(\im Q,\Phi)<0<\lgr(\ker Q,\Phi)$.
\item There exist a projection $Q\in\dP(\R^n)$ and constants $\alpha,\beta\in\R_{>0}$ such that for any $t,s\in\I$, $t\geq s$, and $x\in\im Q$, $y\in\ker Q$ one has
\begin{equation}
\begin{split}\label{eq:classicalhyp}
\abs{\Phi(t,\tmin)x} &\leq \e^{-\alpha(t-s)}\abs{\Phi(s,\tmin)x},\\
\abs{\Phi(t,\tmin)y} &\geq \e^{\beta(t-s)}\abs{\Phi(s,\tmin)y}.
\end{split}
\end{equation}
\end{enumerate}
\end{lemma}

\begin{proof}
(i) $\Rightarrow$ (ii): By \autoref{remark:maxsubspace} there exist subspaces $X,Y$ of appropriate dimension satisfying $\ugr(X,\Phi)=\eugr[k](\Phi)<0<\elgr[k](\Phi)=\lgr(Y,\Phi)$.

(ii) $\Rightarrow$ (i): By \autoref{def:extgrowthrate}, we obtain
\begin{equation*}
\eugr[k](\Phi)\leq\ugr(X,\Phi)<0<\lgr(Y,\Phi)\leq\elgr[n-k](\Phi).
\end{equation*}

(ii) $\Leftrightarrow$ (iii): The subspaces $X$ and $Y$ define a projection by $\im Q\coloneqq X$ and $\ker Q\coloneqq Y$ and vice versa.

(iii) $\Leftrightarrow$ (iv): This follows directly from \autoref{lemma:equivalences}.
\end{proof}

Clearly, the subspaces/projections mentioned in \autoref{lemma:hypgrowth} give rise to so-called invariant projectors, for instance $P(\cdot)=\Phi(\cdot,\tmin)Q\Phi(\tmin,\cdot)$, with corresponding properties, see \cite{Berger2009,Berger2010,Berger2011,Doan2011,Doan2012}.

As can be seen easier from \autoref{lemma:hypgrowth}, \autoref{def:EMD} includes some other (hyperbolicity) notions as special cases. In the finite-time context so far only solution operators $\Phi\in\cLP(\I,\R^n)$ of linear differential equations on a compact time-interval $\I=[\tau,\tau+T]$, $\tau\in\R$, $T\in\R_{>0}$, i.e.\
\begin{equation*}
\dot{x}=A(t)x,
\end{equation*}
where usually $A\in C(\I,\cL(\R^n))$, have been considered. In this case, $\Phi$ regarded as a linear process satisfies even continuous differentiability where we required just Lipschitz continuity in \autoref{def:linprocess}. From \autoref{lemma:equivalences} and \autoref{lemma:hypgrowth} follows directly that, setting $\I=[\tau,\tau+T]$, our definition of EMD coincides with the definition of \emph{M-hyperbolicity} as defined in \cite{Berger2008-1,Doan2011}, \emph{(finite-time) hyperbolicity} as defined in \cite{Berger2009,Berger2011} and \emph{uniform hyperbolicity} as in \cite{Haller2000-1}. By setting $\I=J$ it is obvious that our definition of an EMD corresponds essentially to the finite-time hyperbolicity proposed in \cite{Doan2012}, which generalizes in particular the \emph{nonhyperbolic $(\tau,T)$-dichotomy} as suggested in \cite{Rasmussen2010}. From \autoref{lemma:equivalences} and taking into account the monotonicity preserving property of the logarithm one easily concludes that nonhyperbolic $(\tau,T)$-dichotomy is equivalent to \autoref{def:EMD} with $\I=\set{\tau,\tau+T}$. Another finite-time hyperbolicity notion of the EMD-type which is based only on the start and end time-point can be found in the finite-time Lyapunov exponent approach, see for instance \cite{Haller2001,Shadden2005}. As is shown in \cite{Doan2012} this approach is closely related to our finite-time hyperbolicity notion.

By the nesting property of the extremal growth rates, see \autoref{lemma:gr}, we have that $k\in\set{0,\ldots,n}$ in \autoref{def:EMD} is uniquely defined. Of course, this does not imply that the subspaces/projections mentioned in \autoref{lemma:hypgrowth} are unique. Consequently, the notion of EMD is well-defined, at least up to rank of the projection, in contrast to the definition of a finite-time exponential dichotomy.

In the following, we investigate for given $\Phi\in\cLP(\I,\R^n)$ the associated family of linear processes $(\Phi_{\gamma})_{\gamma\in\R}\in\cLP(\I,\R^n)^{\R}$ which is defined for any $\gamma\in\R$ and $s,t\in\I$ by:
\begin{equation}\label{eq:phigamma}
\Phi_{\gamma}(t,s) \coloneqq \e^{-\gamma(t-s)}\Phi(t,s).
\end{equation}
The motivation to study $\Phi_{\gamma}$ comes from the fact that they arise naturally in the study of linear differential equations
\begin{equation*}
\dot{x} = A(t)x,
\end{equation*}
where, for instance, $A\in C(\I,\cL(\R^n))$, $\I$ an interval, and the corresponding \emph{shifted} linear differential equations for $\gamma\in\R$
\begin{equation}\label{eq:shiftedode}
\dot{x} = (A(t)-\gamma\id_{\R^n})x.
\end{equation}
Clearly, we have by definition that for any $x\in\R^n$ and $(t,s)\in\Ineq$ holds
\begin{equation*}
\Delta(\abs{\Phi_{\gamma}(\cdot,\tmin)x})(t,s)=\Delta(\abs{\Phi(\cdot,\tmin)x})(t,s)-\gamma.
\end{equation*}
As trivial consequences of this observation we obtain the following results. Let $\Phi\in\cLP(\I,\R^n)$ and $\gamma\in\R$. $\Phi_{\gamma}$ admits an EMD on $\I$ with projection $Q\in\dP(\R^n)$ if and only if $\ugr(\im Q,\Phi)<\gamma<\lgr(\ker Q,\Phi)$. Furthermore, the extremal growth rates behave as the logarithmic difference quotient under the exponential weight, i.e.\
\begin{equation}\label{eq:shiftgrowth}
\elgr[k](\Phi_{\gamma})+\gamma=\elgr[k](\Phi)\quad\text{and}\quad\eugr[k](\Phi_{\gamma})+\gamma=\eugr[k](\Phi).
\end{equation}
Consequently, the following characterization of EMD for weighted linear processes holds.

\begin{corollary}\label{cor:shiftedhypercharact}
$\Phi_{\gamma}$ admits an EMD on $\I$ with $k\in\set{0,\ldots,n}$ if and only if $\eugr[k](\Phi)<\gamma<\elgr[n-k](\Phi)$.
\end{corollary}

Since the semimetric $\tilde{d}_{\I}$ depends only on the growth rates, we get the following normalization property: $\tilde{d}_{\I}(\Phi,\Phi_{\gamma})=\gamma$.

\subsection{Spectrum of Linear Finite-Time Processes}

Next, we introduce a finite-time spectral notion which is based on the EMD.

\begin{definition}[Finite-time dichotomy spectrum, cf.\ \cite{Berger2009,Rasmussen2010,Doan2011,Doan2012}]\label{def:emd-spectrum}
We define
\begin{align*}
\spec[\I]\colon\cLP(\I,\R^n)&\to 2^{\R}, & \Phi&\mapsto\set{\gamma\in\R;~\Phi_{\gamma}\text{ does not admit an EMD on }\I},
\end{align*}
and call $\spec[\I](\Phi)$ the \emph{(finite-time dichotomy) spectrum} and $\resolv(\Phi)\coloneqq\R\setminus\spec[\I](\Phi)$ the \emph{(finite-time) resolvent set of $\Phi$}, respectively.
\end{definition}

\begin{remark}\label{remark:spectimeset}
By \autoref{def:emd-spectrum} it is clear that for two compact sets $\J,\I\subset\R$ with $\J\subseteq\I$ and $\Phi\in\cLP(\I,\R^n)$ we have that $\spec[\J](\Phi\bigr|_{\J^2})\subseteq\spec[\I](\Phi)$.
\end{remark}

\begin{theorem}[{Spectral Theorem, \cite[Theorem 10]{Doan2012}, cf.\ also \cite{Berger2009,Doan2011}}] \label{theo:spectrum}
Denote
\begin{equation*}
\set{i_0,\ldots,i_d}\coloneqq\set{j\in\set{0,\ldots,n};~\eugr[j](\Phi)<\elgr[n-j](\Phi)},~i_k<i_{k+1},~k\in\set{0,\ldots,d-1}.
\end{equation*}
Then the spectrum of $\Phi$ is the union of $d$ disjoint compact intervals, i.e.\
\begin{equation*}
\spec[\I](\Phi) = \bigcup_{k=1}^d [\elgr[n-i_{k-1}](\Phi),\eugr[i_k](\Phi)],
\end{equation*}
and we call $[\elgr[n-i_{k-1}](\Phi),\eugr[i_k](\Phi)]$ the \emph{$k$-th spectral interval}.
\end{theorem}

\begin{proof}
This follows directly from \autoref{cor:shiftedhypercharact}, i.e.\ from the fact that
\begin{equation*}
\resolv[\I](\Phi) = \bigcup_{k=0}^n (\eugr[k](\Phi),\elgr[n-k](\Phi)) = \bigcup_{k=0}^d (\eugr[i_k](\Phi),\elgr[n-i_k](\Phi)),
\end{equation*}
since $(\eugr[k](\Phi),\elgr[n-k](\Phi))=\varnothing$ for $k\in\set{0,\ldots,n}\setminus\set{i_0,\ldots,i_d}$, and the nesting property of the extremal growth rates stated in \autoref{lemma:orderedgrowthrates}.
\end{proof}

The assertion of \autoref{theo:spectrum} remains basically true if we require only absolute continuity of $\Phi$. However, in this case, the left-most and right-most spectral intervals may be unbounded; see also the remark after \autoref{def:growthrate} and \cite[Theorem 17]{Berger2009}. \autoref{theo:spectrum} together with \autoref{lemma:orderedgrowthrates} implies that $\spec[\I](\Phi)$ is non-empty for any $\Phi\in\cLP(\I,\R^n)$.

Due to the simple interval-structure of the spectrum where the endpoints of the intervals are the extremal growth rates, we obtain the new result of continuous dependence of $\spec[\I]$ on the linear process from \autoref{lemma:egrcont}.

\begin{proposition}\label{prop:spectrumcont}
The spectrum function $\spec[\I]\colon\left(\cLP(\I,\R^n),d_{\I}\right)\to\left(\cK(\R),\dH\right)$ is Lip\-schitz continuous with Lipschitz constant $1$.
\end{proposition}

Additionally, we obtain a kind of continuity result of the spectrum with respect to the time-set, corresponding to \autoref{lemma:partialcont}.

\begin{corollary}[{\cite[Theorem 17]{Doan2012}}]\label{cor:contintime}
For any $\left(\J_i\right)_{i\in\N}\in\cK(\I)^{\N}$ one has
\begin{equation*}
\lim_{i\to\infty} \dH(\I,\J_i)=0~\Rightarrow~\lim_{i\to\infty} \dH(\spec[\I](\Phi),\spec[\J_i](\Phi\bigr|_{\J_i^2}))=0.
\end{equation*}
\end{corollary}

\begin{proof}
This follows from \autoref{theo:spectrum} and \autoref{lemma:partialcont}.
\end{proof}

\subsection{Robustness of Hyperbolicity}
\label{sec:perturbation}

A desired property of hyperbolicity is its robustness under perturbations, which we establish in a very natural way and for the first time in such a general setting.

\begin{theorem}[Robustness of EMD]\label{theo:robusthyp}
Let $\Phi$ admit an EMD on $\I$ with $k\in\set{0,\ldots,n}$. Then any $\Psi\in\cLP(\I,\R^n)$ with $\tilde{d}_{\I}(\Phi,\Psi)<\min\set{-\eugr[k](\Phi),\elgr[n-k](\Phi)}$ admits an EMD on $\I$ with $k$. For any such perturbation $\Psi$, the extremal subspaces of $\Phi$ can be chosen for the definition of the EMD-projection.
\end{theorem}

\begin{proof}
This is a simple consequence of \autoref{lemma:egrcont} and \autoref{def:metric}, i.e.\ the fact that the extremal growth rates of $\Phi$ and $\Psi$ differ at most as much as $\Phi$ and $\Psi$ do. For any admissible perturbation in the sense of the theorem, the respective growth rates of the subspaces realizing the extremal growth rates (w.r.t.\ $\Phi$) do not cross $0$.
\end{proof}

\begin{corollary}
For given $\I$ the set
\begin{equation*}
\set{\Phi\in\cLP(\I,\R^n);~\Phi\text{ admits an EMD on }\I}
\end{equation*}
is open in $\cLP(\I,\R^n)$ with respect to the topology induced by $d_{\I}$.
\end{corollary}

Next, we are going to show that the estimate in \autoref{theo:robusthyp} is sharp and hence gives the maximal perturbation bound for EMD persistence. This can be interpreted as a hyperbolicity radius or, in case $\Phi$ is an attractive linear process, as a stability radius.

\begin{theorem}[Hyperbolicity radius]\label{theo:stabradiusprocess}
Let $\Phi$ admit an EMD on $\I$. Then
\begin{equation*}
\theta\coloneqq\dist(\set{0},\spec[\I](\Phi))\in\R_{>0}
\end{equation*}
is the largest number such that any $\Psi\in\cLP(\I,\R^n)$ with $\tilde{d}_{\I}(\Phi,\Psi)<\theta$ admits an EMD on $\I$.
\end{theorem}

\begin{proof}
The fact that $\theta$ is a number with the asserted property is clear by \autoref{theo:robusthyp}, thus it remains to show that it is the largest one. This is indeed clear by the normalization property of $\tilde{d}_{\I}$, i.e.\ $\tilde{d}_{\I}(\Phi,\Phi_{\theta})=\theta$ and that $\Phi_{\theta}$ does not admit an EMD on $\I$. The latter is due to the fact that $\theta\in\spec[\I](\Phi)$.
\end{proof}

\section{Linearization of Finite-Time Processes}
\label{sec:linearization}

In this section we study $C^1$-processes $\phi$ on $\I$ and their linearization $\Phi$ along fixed trajectories. In particular, we are interested in local implications of finite-time hyperbolicity of $\Phi$ on the the original process $\phi$. Throughout this section let $\phi\in\cP(\I,\R^n)$ denote a $C^1$-process on $\I$.

\subsection{Linearization and Hyperbolicity of Finite-Time Processes}

Motivated by the classical theory, we introduce the following notion.

\begin{definition}[Linearization]\label{def:linearization}
Let $x\in\R^n$. We define
\begin{align*}
\Phi_{(\tmin,x)}\colon\I\times\I&\to\cL(\R^n), & (t,s)&\mapsto\Phi_{(\tmin,x)}(t,s)\coloneqq\partial_2\phi(t,s,\phi(s,\tmin,x)).
\end{align*}
and call $\Phi_{(\tmin,x)}$ the \emph{linearization of $\phi$ along $\phi(\cdot,\tmin,x)$}.
\end{definition}

\begin{lemma}
For any $x\in\R^n$ the function $\Phi_{(\tmin,x)}$ as defined in \autoref{def:linearization} is a linear process on $\I$.
\end{lemma}

\begin{proof}
The cocycle properties including the invertibility are easily checked with the cocycle properties of $\phi$ and the chain rule of differentiation. Lipschitz continuity of $\Phi_{(\tmin,x)}$ holds by definition.
\end{proof}

\begin{definition}[Finite-time hyperbolicity, attraction and repulsion]
Let $x_0\in\R^n$ and $\Phi_{(\tmin,x_0)}\in\cLP(\I,\R^n)$ be the linearization of $\phi$ along $\phi(\cdot,\tmin,x_0)$. We call $\phi(\cdot,\tmin,x_0)$ \emph{(finite-time) hyperbolic} if $\Phi_{(\tmin,x_0)}$ admits an EMD on $\I$. We call $\phi(\cdot,\tmin,x_0)$ \emph{(finite-time) attractive/repulsive} if $\Phi_{(\tmin,x_0)}$ is attractive/repulsive.
\end{definition}

For an extensive study of finite-time attractivity with respect to the two-point time-set $\I=\set{\tmin,\tmax}$ we refer to \cite{Giesl2012}. To get a better geometrical understanding of the finite-time behavior of trajectories close to some reference trajectory we introduce the following cones, i.e.\ sets that are invariant under scalar multiplication.

\begin{definition}[{Stable/unstable cone, cf.\ \cite[p. 11]{Doan2011}}]\label{def:cones}
We define
\begin{align*}
\Vs\colon\cLP(\I,\R^n)&\to2^{\R^n}, & \Phi&\mapsto\bigcup\set{X\in\Gr(1,\R^n);~\ugr(X,\Phi)<0},\\
\Vu\colon\cLP(\I,\R^n)&\to2^{\R^n}, & \Phi&\mapsto\bigcup\set{X\in\Gr(1,\R^n);~\lgr(X,\Phi)>0},
\end{align*}
where we call $\Vs(\Phi)$ and $\Vu(\Phi)$ the \emph{stable} and \emph{unstable cone of $\Phi$}, respectively.
\end{definition}

Clearly, we can equivalently consider $\Vs(\Phi)$ and $\Vu(\Phi)$ as subsets of $\Gr(1,\R^n)$ (by not taking the union in \autoref{def:cones}). In the following, we identify the set of points lying in a $1$-dimensional subspace $X\subseteq\R^n$ with the point $X\in\Gr(1,\R^n)$. Consequently, we obtain
\begin{equation}\label{eq:identification}
V = \bigcup\set{X\in\Gr(1,\R^n);~X\subseteq V} \cong \set{X\in\Gr(1,\R^n);~X\subseteq V} = \Gr(1,V),
\end{equation}
and in that sense we can speak of a subspace being compact. By the continuity of the growth rate functions for fixed $\Phi$ proved in \autoref{prop:grc}, we obtain the following result.

\begin{lemma}\label{lemma:opencones}
The stable and unstable cones of $\Phi$ are open subsets of $\Gr(1,\R^n)$.
\end{lemma}

The next proposition generalizes \cite[Theorem 14]{Doan2011} by (partially) removing a differentiability assumption on the norm, cf.\ \autoref{remark:generalassumption}.

\begin{proposition}[{cf.\ \cite[Theorem 14]{Doan2011}}]\label{prop:manitocone}
The following statements hold:
\begin{enumerate}[(i)]
\item If $\Phi$ admits an EMD on $\I$ with $Q\in\dP(\R^n)$, then
\begin{align*}
\im Q &\subseteq\Vs(\Phi), & \text{and} && \ker Q &\subseteq\Vu(\Phi).
\end{align*}
\item Suppose there exist subspaces $U_1\subseteq\Vs(\Phi)$ and $U_2\subseteq\Vu(\Phi)$ with $U_1\oplus U_2 = \R^n$. Then $\Phi$ admits an EMD with projection $Q$ defined by $\im Q = U_1$ and $\ker Q = U_2$.
\end{enumerate}
\end{proposition}

\begin{proof}
(i): This follows directly from \autoref{def:EMD} and \autoref{def:cones}.\\
(ii): We need to show that $\ugr(U_1,\Phi)<0<\lgr(U_2,\Phi)$. This is clear by the continuity of the growth rate functions from \autoref{prop:grc} and the compactness of $U_1$ and $U_2$ considered as subsets of $\Gr(1,\R^n)$.
\end{proof}

\begin{remark}\label{remark:definitegrowth}
From the proof of part (ii) we can read off directly that the supremum (infimum) over upper (lower) growth rates with respect to an arbitrary compact subset of $\Vs(\Phi)$ ($\Vu(\Phi)$), considered as subsets of $\Gr(1,\R^n)$, gives a negative (positive) number.
\end{remark}

Finite-time local stable and unstable manifolds are studied for different problem classes in \cite{Haller1998,Haller2000-1,Haller2001,Berger2009}. In these works, the manifolds are introduced as manifolds in the extended phase space, depending on some non-unique extension of a given differential equation on some compact time-interval to the whole real line. The obtained manifolds have, for any given extension, indeed a $C^1$-manifold structure. In \cite[Definition 35]{Duc2008} for two-dimensional ODEs so-called stable and unstable manifolds are introduced, which do not have a manifold structure. However, these objects are defined in an ``intrinsic'' way, by requiring some decay and growth property of trajectories on the compact time-interval with respect to a reference trajectory, respectively. In \cite[Definition 3.1]{Giesl2012} domains of attraction are introduced intrinsically with a decay requirement with respect to the two-point time-set $\I=\set{\tmin,\tmax}$. We take this as a motivation for the next definition.

\begin{definition}[Domains of finite-time attraction/repulsion]
Let $y\in\R^n$. Then we define
\begin{align*}
\Ws_y&\coloneqq\set{x\in\R^n\setminus\set{y};\sup_{(t,s)\in\Ineq}\set{\Delta(\abs{\phi(\cdot,\tmin,x)-\phi(\cdot,\tmin,y)})(t,s)}<0}\cup\set{y},\\
\Wu_y&\coloneqq\set{x\in\R^n\setminus\set{y};\inf_{(t,s)\in\Ineq}\set{\Delta(\abs{\phi(\cdot,\tmin,x)-\phi(\cdot,\tmin,y)})(t,s)}>0}\cup\set{y},
\end{align*}
and call $\Ws_y$ and $\Wu_y$ the \emph{domains of (finite-time) attraction} and \emph{repulsion} with respect to $\phi(\cdot,\tmin,y)$, respectively.
\end{definition}

\begin{remark}
We call $\Ws_y$ and $\Wu_y$ domains of attraction and repulsion, respectively, to emphasize their set structure and to avoid terms like manifold or cone.
\end{remark}

It is easy to see that under a time-dependent coordinate shift $(t,x)\mapsto(t,x-\phi(t,\tmin,y))$ the linearization $\Phi_{(\tmin,y)}$ (along $\phi(\cdot,\tmin,y)$ in the original coordinates) coincides with the linearization $\tilde{\Phi}_{(\tmin,0)}$ (along $\I\times\set{0}$ in the transformed coordinates). Without loss of generality we assume the reference trajectory to be the zero trajectory for the rest of this section.

As a next step, we prove that hyperbolic trajectories have, under some condition on the approximation quality of the process by the linearization, non-empty domains of attraction and repulsion. We are going to show that locally cones and domains look very similar, which is a result of local persistence and we adapt the reasoning that led to the robustness of EMD to the nonlinear case. To this end, we first introduce analogues to the $1$-dimensional growth rates as follows:
\begin{align}
\underline{\mu}\colon\R^n\setminus\set{0}\times\cP(\I,\R^n) &\to\R, & (x,\phi)&\mapsto\inf_{(t,s)\in\Ineq}\Delta(\abs{\phi(\cdot,\tmin,x)})(t,s),\\
\overline{\mu}\colon\R^n\setminus\set{0}\times\cP(\I,\R^n) &\to\R, & (x,\phi)&\mapsto\sup_{(t,s)\in\Ineq}\Delta(\abs{\phi(\cdot,\tmin,x)})(t,s).
\end{align}
Note that for $\Phi\in\cLP(\I,\R^n)$ and $x\in\R^n\setminus\set{0}$ we have $\underline{\mu}(x,\Phi)=\lgr(\linhull\set{x},\Phi)$ and $\overline{\mu}(x,\Phi)=\ugr(\linhull\set{x},\Phi)$. Based on $\underline{\mu}$ and $\overline{\mu}$ we introduce a measure of approximation of the $C^1$-process $\phi$ by the linearization $\Phi$ along the zero reference trajectory
\begin{subequations}\label{eq:semimetric}
\begin{gather}
m\colon\R_{\geq 0}\to\R_{\geq 0},\\
\eta\mapsto
\begin{cases}
0, &\eta=0,\\
\sup_{x\in B[0,\eta]\setminus\set{0}}\max\set{\abs{\underline{\mu}(x,\phi)-\underline{\mu}(x,\Phi)},\abs{\overline{\mu}(x,\phi)-\overline{\mu}(x,\Phi)}}, &\text{otherwise.}
\end{cases}
\end{gather}
\end{subequations}
With this notation at hand, the domains of attraction and repulsion of the zero reference trajectory take the simple form
\begin{align*}
\Ws_0&\coloneqq\set{x\in\R^n\setminus\set{0};~\overline{\mu}(x,\phi)<0}\cup\set{0},\\
\Wu_0&\coloneqq\set{x\in\R^n\setminus\set{0};~\underline{\mu}(x,\phi)>0}\cup\set{0},
\end{align*}
from which the similarity to the stable and unstable cones of \autoref{def:cones} becomes already visible. Roughly speaking, stable and unstable cones correspond to, respectively, stable and unstable subspaces in the infinite-time situation, while domains of attraction and repulsion correspond to local stable and unstable manifolds. The classical Local Stable Manifold Theorem states that, for instance, at a hyperbolic trajectory, there exists a local stable manifold which is tangent to the stable subspace at the trajectory. For the finite-time context, the essential question of this section is to find the relation between cones and domains. The main result \autoref{theo:stabcone} will state much more than just tangency.

Before we prove the main results, we give a sufficient condition for $m$ to be continuous at $0$.

\begin{lemma}\label{lemma:contmdisc}
Let $\Phi\in\cLP(\I,\R^n)$ be the linearization of $\phi$ along $\phi(\cdot,\tmin,0)=0$. If $\I$ is finite then the function $m$ as in Eq.\ \eqref{eq:semimetric} is continuous at $0$.
\end{lemma}

\begin{proof}
By the inverse triangle inequality it suffices to prove
\begin{equation*}
D(t,s,x) \coloneqq \abs{\Delta(\abs{\phi(\cdot,\tmin,x)})(t,s)-\Delta(\abs{\Phi(\cdot,\tmin)x})(t,s)} \xrightarrow{\abs{x}\to 0}0,
\end{equation*}
uniformly in $(t,s)\in\Ineq$. Since $\I$ is finite we have $M\coloneqq\min_{(t,s)\in\Ineq}\abs{t-s}>0$. Observe that
\begin{equation*}
R(t,x)\coloneqq\phi(t,\tmin,x)-\partial_2\phi(t,\tmin,0)x=\phi(t,\tmin,x)-\Phi(t,\tmin)x
\end{equation*}
is continuous and
\begin{equation}\label{eq:Runiform}
\abs{R(t,x)}\xrightarrow{\abs{x}\to 0}0,\quad\text{ uniformly in }t\in\I.
\end{equation}
We estimate
\begin{align*}
D(t,s,x) &= \abs{\frac{\ln\abs{\phi(t,\tmin,x)}-\ln\abs{\phi(s,\tmin,x)}}{t-s} - \frac{\ln\abs{\Phi(t,\tmin)x}-\ln\abs{\Phi(s,\tmin)x}}{t-s}}\\
&\leq \frac{1}{M}\left(\abs{\ln\frac{\abs{\phi(t,\tmin,x)}}{\abs{\Phi(t,\tmin)x}}}+\abs{\ln\frac{\abs{\phi(s,\tmin)x}}{\abs{\Phi(s,\tmin,x)}}}\right)\\
&= \frac{1}{M}\left(\abs{\ln\frac{\abs{\Phi(t,\tmin)x+R(t,x)}}{\abs{\Phi(t,\tmin)x}}}+\abs{\ln\frac{\abs{\Phi(s,\tmin)x+R(s,x)}}{\abs{\Phi(s,\tmin)x}}}\right)\\
&\xrightarrow{\abs{x}\to 0}\frac{1}{M}(\ln 1+\ln 1)=0,
\end{align*}
uniformly in $(t,s)\in\Ineq$ by \autoref{lemma:uniformconttime} and Eq.\ \eqref{eq:Runiform}.
\end{proof}

Note that in the previous lemma we did not impose any extra regularity conditions neither on the norm nor on the process. The next lemma gives sufficient conditions for the ODE case, which requires both additional regularity of the norm and of the process. We state and prove it for the Euclidean norm, making use of the following facts: the Euclidean norm is continuously differentiable on $\R^n\setminus\set{0}$, and the modulus of continuity of its derivative when restricted to a compact domain $D$ not containing the origin is $\omega(t)=t/\alpha$, where $\alpha\coloneqq\min\set{\abs{x};~x\in D}$. The following result can be clearly transferred to other norms by requiring continuous differentiability and a certain behavior of the modulus of continuity of the derivative close to the origin, which will become clear in the course of the proof.

\begin{lemma}\label{lemma:contmcont}
Let $\I$ be a compact interval, $\abs{\cdot}$ denote the Euclidean norm, $\phi\in\cP(\I,\R^n)$ be a $C^2$-process on $\I$ and $\Phi\in\cLP(\I,\R^n)$ be the linearization of $\phi$ along $\phi(\cdot,\tmin,0)=0$. Suppose that
\begin{enumerate}[(i)]
\item $\phi$ and $\Phi$ are continuously differentiable in the first argument, and
\item for $R(t,x)\coloneqq\phi(t,\tmin,x)-\Phi(t,\tmin)x$, $t\in\I$, $x\in\R^n$, one has that $R(t,\cdot)$ together with $\partial_0R(t,\cdot)$ is of class $O(\abs{x}^2)$ for $\abs{x}\to 0$ uniformly in $t\in\I$.
\end{enumerate}
Then the function $m$ as in Eq. \eqref{eq:semimetric} is continuous at $0$.
\end{lemma}

\begin{proof}
As in the previous lemma, we show that $D(t,s,x)\to 0$ as $\abs{x}\to 0$ uniformly in $(t,s)\in\Ineq$, this time by applying the mean value theorem to the continuously differentiable function $t\mapsto\ln\frac{\abs{\phi(t,\tmin,x)}}{\abs{\Phi(t,\tmin)x}}$. To this end, note that for $R(t,x)\coloneqq\phi(t,\tmin,x)-\Phi(t,\tmin)x$, $t\in\I$, $x\in\R^n$, we have that $R(t,\cdot)$ together with $\partial_0R(t,\cdot)$ is of class $o(\abs{x})$ for $\abs{x}\to 0$ uniformly in $t\in\I$ due to the twice continuous differentiability of $\phi$. Besides elementary calculations and estimates, the crucial ingredient of the proof is to show that
\begin{equation}\label{eq:Rapprox}
\norm{\abs{\cdot}'\bigl(\Phi(t,\tmin)x+R(t,x)\bigr)-\abs{\cdot}'\bigl(\Phi(t,\tmin)x\bigr)}\xrightarrow{\abs{x}\to 0}0,
\end{equation}
uniformly in $t\in\I$. To show this, we first observe that there exist constants $\eps,\delta\in\R_{>0}$ such that $\abs{R(t,x)}\leq\eps\abs{x}^2\leq\eps\delta\abs{x}$ whenever $\abs{x}\leq\delta$, due to the aforementioned observation on the convergence of $R$. Since $\Phi$ is invertible and $\I$ is compact, we have that $\alpha\coloneqq\min\norm{\Phi([\I],\tmin)[\S]},\beta\coloneqq\max\norm{\Phi([\I],\tmin)}>0$, where $\alpha$ is the absolute value closest to zero that a trajectory starting on the unit circle attains on $\I$ and, analogously, $\beta$ is the largest such value. We may assume w.l.o.g.\ that $\delta<\alpha/\eps$ and hence $\alpha-\eps\delta>0$. Now choose $\eta<\delta$, then for any $\abs{x}=\eta$ we have that $\abs{\Phi(t,\tmin)x}\in[\alpha\eta,\beta\eta]$ and $\abs{\phi(t,\tmin,x)}\in[(\alpha-\eps\delta)\eta,\beta\eta+\eps\eta^2]$. When restricted to the compact annulus $B[0,\beta\eta+\eps\eta^2]\setminus B(0,(\alpha-\eps\delta)\eta)$, the derivative of the Euclidean norm is uniformly continuous with a modulus of continuity of $\omega(t)=t/((\alpha-\eps\delta)\eta)$. On the other hand, on this annulus and all annuli constructed for smaller $\eta$ we have the quadratic estimate on $R$, yielding $\omega(\eps\eta^2)=\eps\eta^2/((\alpha-\eps\delta)\eta)\xrightarrow{\eta\to 0}0$ and in turn proving Eq.\ \eqref{eq:Rapprox}.
\end{proof}

\subsection{The Local Stable/Unstable Cone Theorem}

In the next two sections we introduce a new, ``intrinsic'' approach to local stable and unstable cones and manifolds, which uses information about $\phi$ on $\I$ only and hence does not rely on classical asymptotic methods. The essential assumption is the continuity of $m$ in $0$.

We define the two functions
\begin{subequations}\label{eq:eta}
\begin{align}
\eta\colon\Gr(1,\R^n)&\to\R_{\geq 0}, & X&\mapsto\inf\set{r\in\R_{>0};~x\in X\cap\S,~rx\notin\Ws_0},\\
\hat{\eta}\colon\Gr(1,\R^n)&\to\R_{\geq 0}, & X&\mapsto\inf\set{r\in\R_{>0};~x\in X\cap\S,~rx\notin\Wu_0}.
\end{align}
\end{subequations}

\begin{theorem}[Local Stable/Unstable Cone Theorem]\label{theo:stabcone}
Let $\Phi\in\cLP(\I,\R^n)$ be the linearization of $\phi$ along $\phi(\cdot,\tmin,0)=0$. If the function $m$ as defined in Eq.\ \eqref{eq:semimetric} is continuous at $0$ then for any $X,Y\in\Gr(1,\R^n)$ with $X\subseteq\Vs(\Phi)$ and $Y\subseteq\Vu(\Phi)$ one has $\eta(X),\hat{\eta}(Y)>0$. Moreover, $\eta$ and $\hat{\eta}$ are bounded away from zero on compact subsets of $\Vs(\Phi)$ and $\Vu(\Phi)$, respectively.
\end{theorem}

\begin{proof}
By \autoref{def:cones} we have for any $X\in\Gr(1,\R^n)$ with $X\subseteq\Vs(\Phi)$ that $\ugr(X,\Phi)<0$. By the continuity assumption on $m$ there exists $\delta\in\R_{>0}$ such that $m(\eta)<-\ugr(X,\Phi)$ for any $\eta\in(0,\delta]$. Then for any $x\in B[0,\delta]\cap X$ we have $\overline{\mu}(x,\phi)<0$ and analogously the assertion for $Y\in\Gr(1,\R^n)$ with $Y\subseteq\Vu(\Phi)$. The second part follows from \autoref{remark:definitegrowth} and the same argument as applied before to single directions $X\in\Gr(1,\R^n)$.
\end{proof}

By the same continuity argument as in \autoref{theo:stabcone} we can find positive radii $\delta\in\R_{>0}$ for the directions in the interior of $V\coloneqq\R^n\setminus(\Vs(\Phi)\cup\Vu(\Phi))$ such that $V\cap B[0,\delta]\cap\Ws_{\Phi},V\cap B[0,\delta]\cap\Wu_{\Phi}=\varnothing$. In other words, the main conclusion from \autoref{theo:stabcone} is that stable and unstable cones (together with their complement) of the linearization $\Phi$ and domains of attraction and repulsion (with their complement) of the process $\phi$, respectively, are locally indistinguishable. This yields a complete local description of the process $\phi$ by its linearization $\Phi$. Additionally, note that this is a pure continuity result and not an implication of hyperbolicity. For the ODE case in $\R^2$ and stronger regularity assumptions a similar approximation result has been proved in \cite[Theorem 44]{Duc2008}.

As a special case we obtain the following result.

\begin{theorem}[Local Stable/Unstable Manifold Theorem]\label{theo:stabmani}
Suppose the assumptions of \autoref{theo:stabcone} are satisfied and let $Q\in\dP(\R^n)$ be a projection such that $\im Q\subseteq\Vs(\Phi)$ and $\ker Q\subseteq\Vu(\Phi)$. Then there exist neighborhoods $U$ and $V$ of the origin such that $\im Q\cap U\subseteq\Ws_0$ and $\ker Q\cap V\subseteq\Wu_0$. Furthermore, for any $t\in\I$ the following equations hold:
\begin{align}\label{eq:tangency}
T_0\phi(t,\tmin,[\im Q]) &= \Phi(t,\tmin)[\im Q], &\text{and} &&
T_0\phi(t,\tmin,[\ker Q]) &= \Phi(t,\tmin)[\ker Q].
\end{align}
Consequently, $\phi([\I],\tmin,[\im Q\cap U])$ and $\phi([\I],\tmin,[\ker Q\cap V])$ can be considered as \emph{finite-time local stable} and \emph{unstable manifolds}, respectively.
\end{theorem}

\begin{proof}
Since $\im Q$ and $\ker Q$ are compact subsets of $\Vs(\Phi)$ and $\Vu(\Phi)$, respectively, \autoref{theo:stabcone} applies and the first part is proved. Furthermore, the tangencies in Eq.\ \eqref{eq:tangency} are easily verified with the definition of the linearization.
\end{proof}

\begin{remark}\label{remark:stabmani1} We want to comment on some issues concerning \autoref{theo:stabmani}.
\begin{enumerate}
\item We would like to point out that \autoref{theo:stabmani} holds for general compact $\I$. So far, finite-time Local Stable Manifold Theorems have been proved only in the ODE case with $\I$ being a compact interval; see \cite{Haller1998,Haller2001,Berger2011} and also \autoref{remark:stabmani2}(2).
\item In the finite-time context, one can consider \autoref{theo:stabcone} and \autoref{theo:stabmani} as robustness results as well. As we proved, locally, the stable and unstable cones (together with the subspaces that they contain) persist under nonlinear perturbations with vanishing first order terms.
\item Despite the lack of structure for the domains of attraction and repulsion themselves, we see that the maximal dimension of manifolds contained in these domains going through the origin corresponds to the indices of the EMD growth rates, i.e.\ to rank and deficiency of the EMD-projection, respectively.
\item Note that by the assumption that $\phi$ be a $C^k$-process on $\I$, $k\in\set{1,2}$, we obtain directly that the extension of $\im Q\cap U$ and $\ker Q\cap U$ by $\phi$ to the extended state space $\I\times\R^n$ gives a $C^k$-manifold in each time-fiber $\set{t}\times\R^n$, $t\in\I$. Evidently, the chart is given by $\phi(t,\tmin,\cdot)$.
\item The function $m$ can be considered as a local measure of nonlinearity of $\phi$, in the sense that $m=0$ if $\phi$ itself is linear and that $m$ takes small values in case $\phi$ is only a small perturbation of a linear process. The more linear $\phi$ becomes, the larger we can choose the radius $\delta\in\R_{>0}$ such that $B(0,\delta)\cap\im Q\subseteq\Ws_0$ and $B(0,\delta)\cap\ker Q\subseteq\Wu_0$. In the ``linear limit'' we recover that $\im Q\subseteq\Ws_0$ and $\ker Q\subseteq\Wu_0$. In this sense, we believe that our version of a local finite-time stable manifold theorem is a very natural one. On the other hand, for fixed nonlinearity, clearly $m$ is an increasing function, i.e.\ the further away we go from the hyperbolic reference trajectory, the weaker we expect the exponential decay/growth to be, until some point where the EMD-subspaces/cones leave the domain of attraction and repulsion, respectively.
\end{enumerate}
\end{remark}

As an easy consequence of \autoref{theo:stabmani} we obtain the following finite-time analogue of the classical Theorem of Linearized Asymptotic Stability. It is a generalization of \cite[Theorem 5.1]{Rasmussen2010} to arbitrary compact time-sets.

\begin{theorem}[Linearized Finite-time Attraction/Repulsion]\label{theo:linstability}
Let $\phi\in\cP(\I,\R^n)$ be a $C^1$-process on $\I$, $x\in\R^n$, $\phi(\cdot,\tmin,x)$ an attractive (repulsive) trajectory and $m$ be continuous in $0$. Then there exists a neighborhood $U$ of $x$ such that $U\subseteq\Ws_x$ ($U\subseteq\Wu_x$, respectively).
\end{theorem}

\subsection{Time-Extensions of Cones and Domains and their Relationship}
\label{sec:hartman-grobman}

Next we investigate the relationship between the cones $\Vs(\Phi)$ and $\Vu(\Phi)$ extended by the linearization $\Phi$ to the extended state space on the one hand, i.e.\
\begin{align*}
\cVs_{\Phi}\coloneqq\Phi(\cdot,\tmin)[\Vs(\Phi)]&=\set{(t,\Phi(t,\tmin)x)\in\I\times\R^n;~(t,x)\in\I\times\Vs(\Phi)},\\
\cVu_{\Phi}\coloneqq\Phi(\cdot,\tmin)[\Vu(\Phi)]&=\set{(t,\Phi(t,\tmin)x)\in\I\times\R^n;~(t,x)\in\I\times\Vu(\Phi)},
\end{align*}
and the domains $\Ws_0$ and $\Wu_0$ extended by $\phi$ on the other hand, i.e.\
\begin{align*}
\cWs_0&\coloneqq\phi(\cdot,\tmin,[\Ws_0])=\set{(t,\phi(t,\tmin,x))\in\I\times\R^n;~(t,x)\in\I\times\Ws_0}, \displaybreak[2]\\
\cWu_0&\coloneqq\phi(\cdot,\tmin,[\Wu_0])=\set{(t,\phi(t,\tmin,x))\in\I\times\R^n;~(t,x)\in\I\times\Wu_0}.
\end{align*}
As usual, we denote by $\cVs_{\Phi}(t)$, $\cVu_{\Phi}(t)$, $\cWs_0(t)$ and $\cWu_0(t)$, $t\in\I$, the $t$-fiber of the respective subsets of the extended state space. The next proposition states that in each $t$-fiber the extended stable and unstable cones are locally contained in the domains of attraction and repulsion, respectively.

\begin{theorem}[Relationship between Extensions]\label{theo:conesdomains}
Suppose the assumptions of \autoref{theo:stabcone} are satisfied. Define the time-extensions of Eq.\ \eqref{eq:eta}
\begin{align*}
\eta\colon\I\times\Gr(1,\R^n)&\to\R_{\geq 0}, & (t,Y)&\mapsto\inf\set{r\in\R_{>0};~y\in Y\cap\S,~ry\notin\cWs_0(t)},\\
\hat{\eta}\colon\I\times\Gr(1,\R^n)&\to\R_{\geq 0}, & (t,Y)&\mapsto\inf\set{r\in\R_{>0};~y\in Y\cap\S,~ry\notin\cWu_0(t)}.
\end{align*}
Then for each $t\in\I$ one has $\eta(t,\cdot)\bigr|_{\cVs_{\Phi}(t)},\hat{\eta}(t,\cdot)\bigr|_{\cVu_{\Phi}(t)}>0$.
\end{theorem}

\begin{proof}
We prove only $\eta(t,\cdot)\bigr|_{\cVs_{\Phi}(t)}>0$, since the second assertion can be proved completely analogously. Let $t\in\I$, $y\in\cVs_{\Phi}(t)\cap\S$. We sketch the idea of the proof first: By the invariance of $\cVs_{\Phi}$ under $\Phi$ it is clear that $\Phi(\tmin,t)y\in\cVs_{\Phi}(\tmin)=\Vs(\Phi)$. Now, consider $x_r\coloneqq\phi(\tmin,t,ry)$, $r\in(0,1]$. To prove that $ry\in\cWs_0(t)$ for sufficiently small $r$, it is sufficient to show that $x_r\in\Ws_0=\cWs_0(\tmin)$. Since $\Vs(\Phi)$ is open by \autoref{lemma:opencones}, our aim is to show that for sufficiently small $r$ we have that $x_r$ is contained in a neighborhood of $1$-dimensional subspaces around $\linhull\set{\Phi(\tmin,t)y}$ which is a subset of $\Vs(\Phi)$. By \autoref{theo:stabcone} we then conclude that for sufficiently small $r$ the vector $x_r$ is contained in the domain of attraction $\Ws_0$. This proves the strict positivity of $\eta(t,\cdot)\bigr|_{\cVs_{\Phi}(t)}$ as claimed.

Thus, it remains to show that $x_r\in\cWs_0(\tmin)$ for sufficiently small $r$. By \autoref{lemma:opencones} there exists $\theta\in\R_{>0}$ such that
\begin{equation*}
B(\Phi(\tmin,t)y,\theta)\subset\cVs_{\Phi}(\tmin)=\Vs(\Phi).
\end{equation*}
Clearly, due to the positive homogeneity of the norm $\abs{\cdot}$ on $\R^n$, the invariance of $\cVs_{\Phi}(\tmin)$ under scalar multiplication and linearity of $\Phi(\tmin,t)$, we find that for all $r\in(0,1]$ holds $B(\Phi(\tmin,t)(ry),r\theta)\subset\cVs_{\Phi}(\tmin)$. By expanding $\phi(\tmin,t,\cdot)$ in $0$ we obtain for all $z\in\R^n\setminus\set{0}$
\begin{equation*}
\abs{\phi(\tmin,t,z)-\Phi(\tmin,t)z}\in o(\abs{z}) \quad\text{for }\abs{z}\to 0.
\end{equation*}
This is equivalent to the fact that for any $\eps\in\R_{>0}$ there exists $\delta\in\R_{>0}$ such that for any $z\in\R^n\setminus\set{0}$ with $\abs{z}\leq\delta$ we have
\begin{equation}\label{eq:smallo}
\abs{\phi(\tmin,t,z)-\Phi(\tmin,t)z}<\eps\abs{z}.
\end{equation}
In other words, for any $\eps\in\R_{>0}$, sufficiently small $\delta\in\R_{>0}$ and $\abs{z}\leq\delta$ we have
\begin{equation*}
\phi(\tmin,t,z)\in B(\Phi(\tmin,t)z,\eps\abs{z})\subseteq B(z,\eps\delta).
\end{equation*}
In particular, choosing $\eps=\theta/2$ we find that for $\delta\in\R_{>0}$ from Eq.\ \eqref{eq:smallo} and consequently for any $ry$ with $r\in[0,\min\set{\delta,1}]$ holds
\begin{equation*}
x_r=\phi(\tmin,t,ry)\in B[\Phi(\tmin,t)(ry),r\theta/2]\subset B(\Phi(\tmin,t)(ry),r\theta)\subset\cVs_{\Phi}(\tmin).
\end{equation*}
Since $B[\linhull\set{\Phi(\tmin,t)(ry)},r\theta/2]\subset \Gr(k,\R^n)$ is closed and hence compact, we know that for sufficiently small $r\in\R_{>0}$ we have $x_r\in\cWs_0(\tmin)$ by \autoref{theo:stabcone}.
\end{proof}

\section{Applications}
\label{sec:applications}

In the following, we want to apply the developed notions and results to ordinary differential equations, which we define without loss of generality globally for the sake of simplicity, i.e.\
\begin{equation}\label{eq:ode}
\dot{x}=f(t,x),
\end{equation}
where $f\in C^1(I\times\R^n,\R^n)$, $I\subseteq\R$ an interval, $\I\subseteq I$ compact and $(f(t,\cdot))_{t\in I}$ is uniformly Lipschitz continuous with Lipschitz constant $L_f$. Thus, Eq.\ \eqref{eq:ode} is well-posed and the solution operator $\phi$ is well-defined. It is well-known that $\phi$ satisfies the conditions of a $C^1$-process. We consider the case that $\abs{\cdot}$ is continuously differentiable on $\R^n\setminus\set{0}$. In particular, this covers all norms induced by the Euclidean inner product and a symmetric positive definite matrix $\Gamma\in\R^{n\times n}$ as considered in \cite{Berger2010,Berger2011}. We fix a solution $\phi(\cdot,\tmin,x)\colon\I\to\R^n$, $x\in\R^n$, and perform a time-dependent coordinate shift of the form $(t,x)\mapsto (t,x-\phi(t,\tmin,x))\eqqcolon(t,y)$. Then in the new coordinates Eq.\ \eqref{eq:ode} takes the form
\begin{equation}\label{eq:linearode}
\dot{y}=\partial_1 f(t,\phi(t,\tmin,x))y+g(t,y)=A(t)y+g(t,y),
\end{equation}
where $A\coloneqq\left(t\mapsto\partial_1 f(t,\phi(t,\tmin,x))\right)\in C(\I,\cL(\R^n))$, $g\in C(\I\times\R^n,\R^n)$ and
\begin{equation*}
g(t,v) = f(t,v+\phi(t,\tmin,x))-f(t,\phi(t,\tmin,x))-\partial_1 f(t,\phi(t,\tmin,x))v,
\end{equation*}
for $t\in\I$ is the nonlinear term. By definition of the derivative we have
\begin{equation*}
g(t,x)/\abs{x}\xrightarrow{\abs{x}\to 0}0
\end{equation*}
for any $t\in\I$. In other words, for any $t\in\I$ and $\eps\in\R_{>0}$ there exists $\delta\in(0,1]$ such that the estimate $\sup\set{\abs{g(t,x)}/\abs{x};~x\in B(0,\delta)\setminus\set{0}}<\eps$ holds. Due to the uniform continuity of $g\bigr|_{\I\times B[0,1]}$, we even obtain that for any $\eps\in\R_{>0}$ there exists $\delta\in(0,1]$ such that
\begin{equation}\label{eq:unismallog}
\sup\set{\frac{\abs{g(t,x)}}{\abs{x}};~x\in B(0,\delta)\setminus\set{0},~t\in\I}<\eps.
\end{equation}
We call
\begin{equation}\label{eq:linode2}
\dot{y}=\partial_1 f(t,\phi(t,\tmin,x))y
\end{equation}
the linearization of \eqref{eq:ode} along $\phi(\cdot,\tmin,x)$. It is well-known that the associated solution operator $\Phi$ of \eqref{eq:linode2}, interpreted as a linear process, is the linearization of $\phi$ along $\phi(\cdot,\tmin,x)$ and that $\Phi$ is continuously differentiable in the first argument. Under the differentiability assumption on the norm the growth rates take the form (as introduced in \cite[Definition 7]{Berger2009})
\begin{equation}
\begin{split}\label{eq:contgrowthrates}
\lgr(X,\Phi) &= \min\set{\frac{(\abs{\Phi(\cdot,\tmin)x})'(t)}{\abs{\Phi(t,\tmin)x}};~t\in\I,~x\in X\cap\S},\\
\ugr(X,\Phi) &= \max\set{\frac{(\abs{\Phi(\cdot,\tmin)x})'(t)}{\abs{\Phi(t,\tmin)x}};~t\in\I,~x\in X\cap\S},
\end{split}
\end{equation}
for $X\in \Gr(k,\R^n)$, which can be seen by \autoref{lemma:equivalences}. From this representation and the chain rule $(\abs{\Phi(\cdot,\tmin)x})'(t)=\abs{\cdot}'(\Phi(t,\tmin)x)\partial_0(\Phi(t,\tmin)x)$ for $t\in\I$ we can see that $d_{\I}$ ($\tilde{d}_{\I}$) can be interpreted as some kind of $C^1$ (semi-)metric for linear solution operators on $\I$.

By classical techniques and Gronwall's lemma one easily establishes that linear right hand sides $A\in C(\I,\cL(\R^n))$ map continuously (with respect to $d_{\I}$) to their unique solution operator $\Phi\in\cLP(\I,\R^n)$. From this observation together with \autoref{theo:robusthyp} we get the following result, which was obtained already in \cite{Berger2010}.

\begin{theorem}[{Robustness of EMD, \cite[Lemma 3]{Berger2010}}]\label{prop:robustlinsys}
Consider
\begin{equation}\label{eq:linode3}
\dot{x}=A(t)x,
\end{equation}
with $A\in C(\I,\cL(\R^n))$. Suppose the associated solution operator $\Phi$ admits an EMD on $\I$. Then there exists $\delta\in\R_{>0}$ such that the solution operator of any $B\in B(A,\delta)$ admits an EMD on $\I$ (with the same extremal projection as $A$).
\end{theorem}

The last result is interesting from the following point of view. When imposed on $\R_{\geq 0}$ the inequalities \eqref{eq:classicalhyp} in \autoref{lemma:hypgrowth} required for an EMD on $\R_{\geq 0}$ correspond to the definition of a so-called \emph{semistrong dichotomy} of \eqref{eq:linode3} on $\R_{\geq 0}$; see \cite{Vinograd1991-1}. \cite{Vinograd1988,Vinograd1991-1} yield that semistrong (exponential) dichotomies on $\R_{\geq 0}$ are robust only in the larger class of general exponential dichotomies, but not within semistrong dichotomies. That means that robustness of EMD cannot be deduced from the classical asymptotic analysis but is a pure finite-time result.

Analogously to the robustness investigation for linear processes in \autoref{sec:perturbation} we now address the question of the stability radius for linear ordinary differential equations given on a compact time-interval.

\begin{definition}[Stability radius]
Suppose Eq.\ \eqref{eq:linode3} with $A\in C(\I,\cL(\R^n))$ generates an attracting solution operator on $I$. Then we define the \emph{stability radius of $A$} by
\begin{equation*}
r(A)\coloneqq\inf\set{\norm{B}_{\infty};~B\in C(\I,\cL(\R^n)),~(A+B)\text{ is not attracting on $\I$}}.
\end{equation*}
\end{definition}

To calculate the stability radius of a given $A$ we make use of an elementary result which can be found in \cite{Coppel1978} and which specializes to the following result.

\begin{proposition}[{cf.\ \cite[Proposition 1, p.\ 2]{Coppel1978}}]
Consider
\begin{equation*}
\dot{x}=A(t)x,
\end{equation*}
with $A\in C(\I,\cL(\R^n))$. Suppose the associated solution operator $\Phi$ is attractive on $\I$, i.e.\ $\Phi$ admits an EMD on $\I$ with the trivial projection $\id_{\R^n}$ and $\ugr(\R^n,\Phi)<0$. Then for any $B\in C(\I,\cL(\R^n))$ with $\norm{B}_{\infty}\leq\abs{\ugr(\R^n,\Phi)}\eqqcolon\delta$ the solution operator $\Psi$ of the perturbed ODE
\begin{equation*}
\dot{y}=(A(t)+B(t))y
\end{equation*}
satisfies $\ugr(\R^n,\Psi)\leq\ugr(\R^n,\Phi)+\delta=0$.
\end{proposition}

In other words, $-\ugr(\R^n,\Phi)>0$ is a lower bound on the stability radius around $A$ in $C(\I,\cL(\R^n))$. The fact that it is also an upper bound follows directly from the well-known correspondence of shifted ODEs and weighted processes (cf.\ Eq.\ \eqref{eq:shiftedode}), the definition of spectrum based on the weights/shifts and the fact that for operator norms induced by a vector norm the identity has norm 1. In summary, by a completely analogous argumentation as in the proof of \autoref{theo:stabradiusprocess} we find that
\begin{equation*}
\dot{y}=(A(t)-\ugr(\R^n,\Phi)\id_{\R^n})y
\end{equation*}
is not attractive. Thus, we obtained that the stability radius of an attractive linear ODE given on $\I$ and the (pseudo-)stability radius of its associated solution operator coincide.

\begin{theorem}[Finite-time Stability Radius]
Let $A\in C(\I,\cL(\R^n))$ and let the associated solution operator $\Phi$ be attractive. Then
\begin{equation*}
r(A) = -\ugr(\R^n,\Phi).
\end{equation*}
\end{theorem}

Another consequence of EMD-robustness deals with linearizations. The next lemma states that the linearization depends continuously on the initial value of the trajectory along which we linearize.

\begin{lemma}\label{lemma:contlinear}
Consider Eq.\ \eqref{eq:ode} and let $\phi$ denote the associated solution operator. Then the function
\begin{align*}
\R^n &\to C(\I,\cL(\R^n)), & x&\mapsto\partial_1f(\cdot,\phi(\cdot,\tmin,x)),
\end{align*}
is continuous.
\end{lemma}

\begin{proof}
Let $t\in\I$, $x\in\R^n$ and $\eps\in\R_{>0}$. By the $C^1$ assumption on $f$ we have in particular that there exists $\delta_1\in\R_{>0}$ such that
\begin{equation*}
\abs{\phi(t,\tmin,x)-\phi(t,\tmin,y)}<\delta_1\Rightarrow\norm{\partial_1f(t,\phi(t,\tmin,x))-\partial_1f(t,\phi(t,\tmin,y))}<\eps.
\end{equation*}
Next note that $y\mapsto\phi(\cdot,\tmin,y)\in C(\I,\R^n)$ is uniformly continuous on any bounded subset $U\subseteq\R^n$ with $x\in U$, i.e.\ there exists $\delta_2\in\R_{>0}$ such that for any $y\in U$ we have
\begin{equation*}
\abs{x-y}<\delta_2\quad\Longrightarrow\quad\norm{\phi(\cdot,\tmin,x)-\phi(\cdot,\tmin,y)}_{\infty}<\delta_1.
\end{equation*}
Combining the two continuity observations we obtain that for any $y\in\R^n$ we have
\begin{equation*}
\abs{x-y}<\delta_2\quad\Longrightarrow\quad\norm{\partial_1f(\cdot,\phi(\cdot,\tmin,x))-\partial_1f(\cdot,\phi(\cdot,\tmin,y))}_{\infty}<\eps.\qedhere
\end{equation*}
\end{proof}

Consequently, robustness of EMD carries over to initial values.

\begin{corollary}[{\cite[Theorem 5]{Berger2010}}]\label{cor:EMDrobustnonlinear}
Consider Eq.\ \eqref{eq:ode}, let $\phi$ be the associated solution operator and $\phi(\cdot,\tmin,x)$ be a hyperbolic/attractive/repulsive reference solution. Then there exists $\delta\in\R_{>0}$ such that for any $y\in B(x,\delta)$ the trajectories $\phi(\cdot,\tmin,y)$ are, respectively, hyperbolic/attractive/repulsive (with respect to the same corresponding subspaces).
\end{corollary}

The next step is to show that \autoref{lemma:contmcont} applies. Therefore, it remains to establish the necessary order of convergence for the linearization error. To this end, consider Eq.\ \eqref{eq:ode} and assume that $f\in C^{0,2}(I\times\R^n,\R^n)$, i.e.\ $f$ is continuous in the first argument and twice continuously differentiable in the second argument. By Taylor's Theorem, the estimate \eqref{eq:unismallog} on the nonlinear term $g$ improves as follows: there exist $\eps,\delta\in\R_{>0}$ such that
\begin{equation}\label{eq:unismallog2}
\sup\set{\frac{\abs{g(t,x)}}{\abs{x}^2};~x\in B(0,\delta)\setminus\set{0},~t\in\I}<\eps.
\end{equation}

\begin{lemma}\label{lemma:c0cont}
Let $\phi$ be the solution operator of \eqref{eq:linearode} with $f\in C^{0,2}(I\times\R^n,\R^n)$ and $\Phi$ be the solution operator of the linearization \eqref{eq:linode2} along the reference solution $\phi(\cdot,\tmin,x)=0$. Then the function
\begin{align*}
\R_{\geq 0}&\to\R_{\geq 0},\\ \eta&\mapsto
\begin{cases}
0, & \eta=0,\\
\sup\set{\norm{\phi(\cdot,\tmin,y)-\Phi(\cdot,\tmin)y}_{\infty};~y\in B[0,\eta]}, &\text{otherwise},
\end{cases}
\end{align*}
is continuous in $0$. Moreover, it is of class $\mathcal{O}(\eta^2)$ for $\eta\to 0$.
\end{lemma}

\begin{proof}
Let $\eta\in\R_{>0}$ and $y\in\R^n$, $\abs{y}\leq\eta$. Integrating Eqs.\ \eqref{eq:linearode} and \eqref{eq:linode2}, we calculate for any $t\in\I$
\begin{align*}
\abs{\phi(t,\tmin,y)-\Phi(t,\tmin)y}&=\abs{\int_{\tmin}^t A(s)\left(\phi(s,\tmin,y)-\Phi(t,s)y\right)+g(s,\phi(s,\tmin,y))\d s}\\
&\leq\int_{\tmin}^t\abs{A(s)(\phi(s,\tmin,y)-\Phi(s,\tmin)y)}\d s+\displaybreak[2]\\
&\phantom{\leq} + \int_{\tmin}^t \abs{g(s,\phi(s,\tmin,y))-g(s,\Phi(s,\tmin)y)}\d s+\displaybreak[2]\\
&\phantom{\leq} + \int_{\tmin}^t\abs{g(s,\Phi(s,\tmin)y)}\d s\displaybreak[2]\\
&\leq (2\norm{A}_{\infty}+L_f)\int_{\tmin}^t\abs{\phi(s,\tmin,y)-\Phi(s,\tmin)y}\d s+\displaybreak[2]\\
&\phantom{\leq} + (\tmax-\tmin)\eps C_{\Phi}\eta^2,
\end{align*}
where $C_{\Phi}\coloneqq\norm{\Phi(\cdot,\cdot)}_{\infty}<\infty$ and $\eps\in\R_{>0}$ satisfies Eq.\ \eqref{eq:unismallog2} with $\delta\coloneqq C_\Phi\eta$. With the abbreviation
\begin{equation*}
C(\eta)\coloneqq(\tmax-\tmin)\eps C_{\Phi}\eta^2\e^{(2\norm{A}_{\infty}+L_f)(\tmax-\tmin)}\in\mathcal{O}(\eta^2),
\end{equation*}
the uniform estimate from above and Gronwall's lemma we obtain
\begin{equation*}
\sup\set{\norm{\phi(\cdot,\tmin,y)-\Phi(\cdot,\tmin)y}_{\infty};~y\in B[0,\eta]}\leq C(\eta)
\end{equation*}
and the assertion is proved.
\end{proof}

\begin{lemma}\label{lemma:c1cont}
Let $\phi$ be the solution operator of Eq.\ \eqref{eq:linearode} with $f\in C^{0,2}(I\times\R^n,\R^n)$ and $\Phi$ be the solution operator of the linearization \eqref{eq:linode2} along the reference solution $\phi(\cdot,\tmin,x)=0$. Then the function
\begin{align*}
\R_{\geq 0}&\to\R_{\geq 0},\\ \eta&\mapsto
\begin{cases}
0, & \eta=0,\\
\sup\set{\norm{\partial_0\phi(\cdot,\tmin,y)-\partial_0\Phi(\cdot,\tmin)y}_{\infty};~y\in B[0,\eta]}, &\text{otherwise},
\end{cases}
\end{align*}
is continuous in $0$. Moreover, it is of class $\mathcal{O}(\eta^2)$ for $\eta\to 0$.
\end{lemma}

\begin{proof}
Let $\eta\in\R_{>0}$ and $y\in\R^n$, $\abs{y}\leq\eta$. We calculate
\begin{align*}
\abs{\partial_0\phi(t,\tmin,y)-\partial_0\Phi(t,\tmin)y} &= \abs{A(t)(\phi(t,\tmin,y)-\Phi(t,\tmin)y)+g(t,\phi(t,\tmin,y))}\\
&\leq\norm{A}_{\infty}\abs{\phi(t,\tmin,y)-\Phi(t,\tmin)y}+\\
&\phantom{\leq}+\abs{g(t,\phi(t,\tmin,y))-g(t,\Phi(t,\tmin)y)}+\\
&\phantom{\leq}+\abs{g(t,\Phi(t,\tmin)y)}\\
&\leq\norm{A}_{\infty}C(\eta)+(\norm{A}_{\infty}+L_f)C(\eta)+\eps C_{\Phi}\eta^2\\
&\leq (2\norm{A}_{\infty}+L_f)C(\eta)+\eps C_{\Phi}\eta^2,
\end{align*}
where we used the notation from the proof of \autoref{lemma:c0cont}. This proves the assertion.
\end{proof}

In summary, we found sufficient conditions to ensure that the function $m$ as defined in Eq. \eqref{eq:semimetric} is continuous at $0$ and hence \autoref{theo:stabcone}, \autoref{theo:stabmani} and \autoref{theo:linstability} apply.

\begin{remark}\label{remark:stabmani2}
We want to comment on some issues concerning the application of the results in \autoref{sec:linearization} to ODEs.
\begin{enumerate}
\item Concerning \autoref{theo:stabmani}, note that EMD-subspaces from the starting time-fiber $\set{\tmin}\times\R^n$ evolve non-linearly under $\phi$. Their extensions via $\phi$ considered as subsets in the extended state space $\I\times\R^n$ are $C^1$-manifolds since $\phi(\cdot,\tmin,\cdot)$ is continuously differentiable; see, for instance, \cite[Theorem 9.2]{Amann1990}.
\item Our version of \autoref{theo:stabmani} applied to the ODE situation extends some previous work on that topic:
in \cite{Haller1998,Haller2001,Berger2011} the ODE is first extended to the real line and the desired manifolds are then obtained as the local stable and unstable manifolds of solutions that are hyperbolic on $\R$ in the sense that they admit an exponential dichotomy. Note that our proof, in principle, does not restrict to norms induced by an inner product weighted with a symmetric, positive definite matrix, as it is done in \cite{Berger2011} (cf.\ also \cite[Remark 4(ii)]{Berger2011}). In \cite{Duc2008} an ``intrinsic'' proof is presented which is tailored for the case $\R^n=\R^2$. There it is shown, that the domains of attraction and repulsion are not empty under appropriate conditions.
\end{enumerate}
\end{remark}

\begin{remark}
Since our hyperbolicity notion is based fundamentally on the monotonicity of the norm of trajectories, a finite-time conjugacy between the linear process $\Phi$ and the general process $\phi$ should preserve the type of monotonicity of trajectories. Roughly speaking and supposing that the assumptions are satisfied for $\I$, \autoref{theo:conesdomains} can therefore be interpreted as a finite-time Hartman-Grobman-like theorem in the following informal sense: as demonstrated in \cite[p. 546]{Siegmund2002-2} the function
\begin{align*}
H\colon\I\times\R^n&\to\R^n, & (t,x)&\mapsto\phi(t,\tmin,\Phi(\tmin,t)x),
\intertext{with fiberwise inverse}
H(t,\cdot)^{-1}\colon\R^n&\to\R^n, & x&\mapsto\Phi(t,\tmin,\phi(\tmin,t)x),
\end{align*}
maps trajectories of $\Phi$ to trajectories of $\phi$ homeomorphically and is therefore a candidate for a (nonautonomous) topological conjugacy between $\Phi$ and $\phi$ with respect to the two zero reference solutions. Now consider restrictions of $H$ to $\Vs(\Phi)$ and $\Vu(\Phi)$ (or compact subsets $S\subset\Vs(\Phi)$ and $U\subset\Vu(\Phi)$ considered as subsets of $\Gr(1,\R^n)$). By \autoref{theo:conesdomains} we obtain for $y$ from the respective set with $\abs{y}$ sufficiently small, that $H$ preserves the monotonicity type of $\Phi(\cdot,\tmin)y$, but, in general, not the exponential rate. The task of finding a monotonicity preserving (nonautonomous) topological conjugacy for a whole neighborhood seems to be futile since too restrictive. In particular it is unclear how ``monotonicity preservation'' should be applied to trajectories with ``indefinite'' monotonicity behavior.
\end{remark}

\section{Conclusions}
\label{sec:conclusion}

In this work we introduced an abstract framework for what is often called ``finite-time dynamics'' by adapting the notion of (invertible) processes to the situation of a compact time set $\I$. Based on the notion of growth rates which we defined independently from any hyperbolicity notion we managed to give a coherent, unified and comprehensive presentation of the theory of linear analysis of finite-time processes on $\I$. Due to the strong ambition to establish continuity of the involved functions, we obtained new and simpler arguments to prove partially known results, sometimes getting rid of technical assumptions such as differentiability or the finiteness of the time set.

Evidently and once more shown by our results, finite-time hyperbolic trajectories play an important role in the local analysis of finite-time dynamics. Therefore, to have sufficient, possibly computable criteria for hyperbolic solutions is necessary for the application of the hyperbolicity concept to real-life problems. For two-dimensional Hamiltonian systems the existence of a hyperbolic solution in the neighborhood of a curve of so-called instantaneous stagnation points under some velocity bound on the governing vector field is proved in \cite{Haller1998}, whereas \cite{Duc2011} proved recently the existence in the neighborhood of an approximate solution. In \cite{Doan2011} it is shown that row diagonal dominance of the linear right hand side implies an EMD on $\I$. Other sufficient criteria are based on a partitioning approach \cite{Haller2001-1,Berger2011,Doan2011}.

We would like to emphasize that our way of reasoning is led by the idea to consider finite-time analysis as the analysis of ``unextendible'' systems, thereby restricting any argumentation to the time set $\I$. This allows for a consequent use of continuity arguments. However, other methods of analysis and proof may become available and reasonable when approximating infinite-time dynamical systems on bounded time-sets.

\subsection*{Acknowledgments}

The author is indebted to Maik Gröger, Sascha Trostorff and Marcus Waurick for several enlightening discussions. He also would like to thank Stefan Siegmund for introducing him to finite-time dynamics and for a couple of comments, Martin Rasmussen for bringing \cite{Berger2010} to his attention, and an anonymous referee for a thorough review.

\bibliographystyle{plain}

\end{document}